\documentclass[11pt]{amsart}
\usepackage[T1]{fontenc}
\usepackage[latin9]{inputenc}
\usepackage{babel}
\usepackage{verbatim}
\usepackage{mathrsfs}
\usepackage{amstext}
\usepackage{amsthm}
\usepackage{amssymb}
\usepackage{pgfplots}
\usepackage{float}
\usepackage{esint}
\usepackage{enumerate}
\usepackage{fancyhdr}

\usepackage[colorlinks=true]{hyperref}\usepackage[shortlabels]{enumitem}\hypersetup{colorlinks=true,linkcolor=purple, citecolor=blue}

\usepackage[markup=underlined,commandnameprefix=always]{changes}
\definechangesauthor[color=red]{QH}

\makeatletter
\numberwithin{equation}{section}
\numberwithin{figure}{section}
\theoremstyle{plain}
\newtheorem{thm}{\protect\theoremname}[section]
\theoremstyle{definition}
\newtheorem{rem}[thm]{\protect\remarkname}
\theoremstyle{definition}

\theoremstyle{plain}
\newtheorem{prop}[thm]{\protect\propositionname}
\theoremstyle{plain}
\newtheorem{lem}[thm]{\protect\lemmaname}

\newtheorem{claim}[thm]{Claim}

\theoremstyle{plain}

\theoremstyle{plain}
\newtheorem{cor}[thm]{\protect\corollaryname}

\theoremstyle{definition}

\theoremstyle{definition}

\theoremstyle{definition}
\newtheorem{examplex}[thm]{\protect\examplename}
\newenvironment{example}
{\pushQED{\qed}\examplex}
{\popQED\endexamplex}

\usepackage{tikz}
\usetikzlibrary{shapes,arrows}
\usepackage{verbatim}
\usepackage{amsthm}
\usepackage{amstext}

\newcommand{\eps}{\varepsilon}
\newcommand{\Om}{\Omega}

\newcommand{\R}{\mathbb{R}}
\newcommand{\N}{\mathbb{N}}
\newcommand{\Z}{\mathbb{Z}}
\newcommand{\mF}{\mathcal{F}}\newcommand{\mS}{\mathcal{S}}

\newcommand{\C}{\mathbb{C}}
\newcommand{\dd}{\mathbf{d}}

\newcommand{\mP}{\mathcal{P}}

\newcommand{\bp}{\begin{proof}}
	\newcommand{\ep}{\end{proof}}
\newcommand{\ummdim}{\overline{\mdim}_{\mathrm M}}
\newcommand{\lmmdim}{\underline{\mdim}_{\mathrm M}}
\newcommand{\mmdim}{\mdim_{\mathrm M}}
\newcommand{\mhdim}{\mdim_{\mathrm H}}
\newcommand{\madim}{\mdim_{\mathrm A}}
\DeclareMathOperator{\pdim}{pdim_{\mathrm A}}
\newcommand{\mldim}{\mdim_{\mathrm L}}
\newcommand{\mqadim}{\mdim_{\mathrm {qA}}}
\newcommand{\udim}{\overline{\dim}_{\mathrm M}\,}
\newcommand{\ldim}{\underline{\dim}_{\mathrm M}\,}
\newcommand{\bdim}{\dim_{\mathrm M}}

\newcommand{\htop}{h_{\mathrm{top}}}

\newcommand{\adim}{\dim_{\mathrm A}}
\newcommand{\ladim}{\dim_{\mathrm L}}

\newcommand{\supp}{\mathrm{supp}}

\DeclareMathOperator{\diam}{diam}
\DeclareMathOperator{\mdim}{mdim}

\newcommand{\sep}{\mathrm{sep}}

\DeclareMathOperator{\dist}{dist}

\makeatother
\providecommand{\conjecturename}{Conjecture}
\providecommand{\corollaryname}{Corollary}
\providecommand{\definitionname}{Definition}
\providecommand{\examplename}{Example}
\providecommand{\lemmaname}{Lemma}
\providecommand{\problemname}{Problem}
\providecommand{\propositionname}{Proposition}
\providecommand{\remarkname}{Remark}
\providecommand{\theoremname}{Theorem}
\providecommand{\taskname}{Task}

\begin{document}

 \author{Qiang Huo$^1$}
 \address{$^1$ School of Mathematical Sciences, University of Science and Technology of China, Hefei, Anhui, 230026, China}
 \email{qianghuo@ustc.edu.cn}

 \author{Adam \'Spiewak$^2$}
\address{$^2$ Institute of Mathematics of the Polish Academy of Sciences, ul. \'Sniadeckich 8, Warsaw, 00-656, Poland}
\email{ad.spiewak@gmail.com}

 \keywords{Mean Assouad dimension, mean Assouad spectrum, dimension interpolation, Bedford--McMullen carpet systems, topological conditional entropy, bi-Lipschitz equivalence.}
 \subjclass[2020]{28A80, 37B40, 37C45, 37B10.}

 \thanks{Q. Huo was partially supported by the China Postdoctoral Science Foundation 2025M773065, the National Key Research and Development Program of China 2024YFA1013600 and  Fundamental Research Funds for the Central Universities WK0010250102. A. \'Spiewak was partially supported by the National Science Centre (Poland) grant 2020/39/B/ST1/02329. The work began while the first-named author visited the Gda\'{n}sk branch of Institute of Mathematics of the Polish Academy of Sciences. He is grateful to the institute for its hospitality and warm atmosphere. Both authors are grateful to Yonatan Gutman, Istv\'{a}n Kolossv\'{a}ry, Ruxi Shi for useful discussions. We also thank Alex Rutar for bringing the reference \cite{OR25branching} to our attention. }

\title{Mean Assouad dimension and spectrum, with applications to infinite dimensional fractals}

\date{\today}
\maketitle
 \begin{abstract}
     We introduce the \textit{mean Assouad dimension} of a dynamical system, motivated by the Assouad dimension in fractal geometry. Using dimension interpolation, we further define the \textit{mean Assouad spectrum}. This provides a new family of bi-Lipschitz invariants of dynamical systems. We study its basic properties and calculate it for several classes of dynamical systems. As an application, we determine explicit formulae for the mean Assouad dimension and spectrum of infinite-dimensional Bedford--McMullen carpet systems, contributing to the program of studying infinite dimensional fractals, initiated recently by Tsukamoto.
 \end{abstract}

\tableofcontents

\section{Introduction}

\subsection{Background}
The classical dimension theory focuses on the complexity of sets and measures in terms of various dimensions. The most familiar are \textit{global} quantities like the Hausdorff and Minkowski  (box-counting) dimensions, which measure a set by covers of varying diameters and uniform diameters respectively. More recently, \textit{extreme} quantities like the lower and Assouad dimensions, which quantify the thinnest and thickest parts of a set across arbitrary pair of scales respectively, are intensively studied. One has the following relationship:
\begin{equation}\label{dimension comparison}
    \text{lower dimension} \leq \text{Hausdorff dimension}\leq\text{Minkowski dimension}\leq \text{Assouad dimension}.
\end{equation}
It is well-known that the phenomenon of `dimension gap' occurs for certain fractal sets. For instance, Mackay \cite{Mac11} and Fraser \cite{Fra14} proved that for planar Bedford--McMullen carpets \cite{Bed84,Mc84} with non-uniform fibres, all inequalities in \eqref{dimension comparison} are strict. Recently, Jurga \cite{Jur23} constructed examples of dynamically invariant sets (i.e. invariant under certain non-conformal expanding maps) for which the upper and lower Minkowski dimensions do not coincide.
Nowadays, a fruitful approach using \textit{dimension interpolation} \cite{Fra21b} has been developed to better understand the gap between various dimensions. 
The starting point was the Assouad spectrum and its dual spectrum introduced by Fraser and Yu \cite{FY18}, with the following relations:
\begin{equation*}
    \text{upper Minkowski dimension} \leq \text{Assouad spectrum}\leq\text{Assouad dimension},
\end{equation*}
\begin{equation*}
    \text{lower dimension} \leq \text{lower spectrum}\leq\text{lower Minkowski dimension}.
\end{equation*}
The Assouad spectrum resembles the definition of the Assouad dimension by fixing the two scales with a parameter. These and similar dimension spectra have been extensively studied in the literature, see e.g. \cite{Ban23,BC23,BF23,BF24,BK24,BFKR24,BR22,CFY22,Feng24,FY18b,GHM21,Rut24}.

In 1999, Gromov \cite{G} introduced a new invariant of dynamical systems, called \textit{mean dimension}, connecting dimension theory with dynamics.  Mean dimension was systematically developed by Lindenstrauss and Weiss \cite{LW00} and Lindenstrauss \cite{Lin99}, and received widespread attention from various fields, including equivariant embedding problems \cite{Gut15,GT20,Levin23} (see \cite{GLT16,GQT19} for similar results on $\Z^k$-actions; see also \cite{Nar24} for recent developments on amenable group actions), its connection to rate distortion theory \cite{LT18,LT19,Tsu25RandomBrody}, crossed product C$^*$ algebras \cite{ENDuke,Niucomparisonradius} and others. In addition to the topological mean dimension introduced by Gromov, a metric-dependent quantity named metric mean dimension was introduced in \cite{LW00}, combining the ideas of topological entropy and Minkowski dimension. The metric mean dimension turns out to be a useful tool to obtain the upper bound on mean topological dimension for some dynamical systems arising from geometric analysis and complex geometry 
(\cite{Yang-Mills,Tsu18Brody,Tsu22}), and recently it is intensively studied on its own, e.g. in the context of ergodic theory \cite{LT18, Shi22,CPV24}, infinite-dimensional fractal geometry \cite{Tsu25carpets, HuoSponges24,LL25} and analog compression \cite{mmdimcompress}. Further, Lindenstrauss and Tsukamoto \cite{LT19} defined the mean Hausdorff dimension, a dynamical analogue of the Hausdorff dimension. We refer the reader to \cite[Proposition 3.2]{LT19} for a comparison between these dynamical dimensions. Recently, Liu, Selmi and Li \cite{LSL24} introduced and studied the mean $\Psi$-intermediate dimensions of dynamical systems, which is a continuously parameterized family of dynamical dimensions varying between the mean Hausdorff dimension and the metric mean dimension.

\subsection{Results}
The aim of this paper is to establish a way to dynamicalize the Assouad type dimensions and spectra, introducing their mean dimension versions. Our first motivation is introducing the methods of Assouad dimension to the dynamical systems theory. As our main application, we characterize the mean Assouad dimension of Bedford-McMullen carpet systems, contributing to the program of studying infinite-dimensional fractals, initiated recently by Tsukamoto \cite{Tsu25carpets}, who calculated their mean Hausdorff and metric dimensions. We also expect the mean Assouad dimension to be useful in the mean dimension-based embedding theory \cite{GT20, GQT19}, as the classical Assouad dimension governs the existence of almost bi-Lipschitz embeddings, see e.g. \cite{Olson02, Rob11, Fra21}. With that in mind, we calculate the mean Assouad dimension of the space of band-limited functions with the shift dynamics, which played a crucial role in the Gutman and Tsukamoto's proof of the embedding theorem for mean topological dimension \cite{GT20}. We also provide a comparison with other possible defintions of mean Assouad dimension.
 
\subsection{Organization of the paper}

We begin with the definition of mean Assouad dimension and spectra in Subsection \ref{defn:mean Assouad dim}. Some of their basic properties are collected in Section \ref{results}. This includes its bi-Lipschitz invariance, basic inequalities between metric mean dimension and mean Assouad spectrum, as well as formulae in terms of local stable sets and non-wandering sets. In Section \ref{sec: examples} we calculate mean Assouad dimension of some examples, including the space of band-limited functions. Section \ref{sec:carpet systems} contains the calculation of the mean Assouad dimension and spectrum of Bedford--McMullen carpet systems (Theorem \ref{theorem: mean Assouad dimension of carpet systems}). In Section \ref{sec: other def} we make a comparison with other possible definitions of mean Assouad dimension.

\section{Mean Assouad dimension and spectra}

In this section we give the definitions of the main notions introduced in this paper - mean Assouad dimension and spectra. We begin with the review of fractal dimensions and spectra, as well as topological entropy and metric mean dimension, which will serve as basis for the new dynamical invariants.

\subsection{Fractal dimensions and dimension spectra}
Throughout the paper, $\log$ denotes the natural logarithm. Let $(X,d)$ be a compact metric space. For $\varepsilon>0$ and a subset set $K$ of $X$, we define the $\varepsilon$-\textbf{covering number} $N_d(K,\varepsilon)$ as the minimum $n\geq 1$ such that there exists an open cover $\{U_1,\ldots,U_n\}$ of $K$ satisfying $\diam U_i\leq\varepsilon$ for each $1\leq i\leq n$.
We also define the $\varepsilon$-\textbf{separating number} $\widehat{N}_d(K,\varepsilon)$ as the maximum $n\geq 1$ such that there exist $x_1,\ldots,x_n\in K$ satisfying $d(x_i,x_j)\geq\varepsilon$ for all $i\neq j$.
We define the \textbf{upper and lower Minkowski dimensions} (also called \textbf{box-counting dimensions}) of $(X,d)$ by
\begin{equation*}
    \udim(X,d)=\limsup\limits_{\varepsilon\to0}\frac{\log N_d(X,\varepsilon)}{\log(1/\varepsilon)}, \ldim(X,d)=\liminf\limits_{\varepsilon\to0}\frac{\log N_d(X,\varepsilon)}{\log(1/\varepsilon)}.
\end{equation*}
If these two values coincide, the common value is called the \textbf{Minkowski dimension} of $(X,d)$ and denoted by $\bdim(X,d)$.

The \textbf{Assouad dimension} of $(X,d)$ is defined by
\begin{equation*}
        \adim (X,d)=\inf\left\{s>0: \underset{C>0}{\exists}\ \underset{0<\rho<r}{\forall}\ \underset{x\in X}{\forall}\  N_d(B_d(x,r),\rho)\leq C(r/\rho)^s\right\},
\end{equation*}
where $B_d(x,r)=\{y\in X: d(x,y)\leq r\}$.

To better understand the gap between the Minkowski and Assouad dimensions, Fraser and Yu \cite{FY18} defined the \textbf{Assouad spectrum} of $(X,d)$ to be
\begin{equation*}
        \adim^{\theta} (X,d)=\inf\left\{s>0: \underset{C>0}{\exists}\ \underset{0<r<1}{\forall}\ \underset{x\in X}{\forall}\  N_d(B_d(x,r),r^{1/\theta})\leq C(r/r^{1/\theta})^s\right\}
\end{equation*}
for all $\theta\in(0,1)$.

\subsection{Topological entropy and metric mean dimension}\label{defn:topological conditional entropy}
By a \textbf{topological dynamical system} (\textit{TDS} for short) we understand a triple $(X,T,d)$ where $(X,d)$ is a compact metric space\footnote{Sometimes we omit $d$ from the notation.} and $T:X\to X$ is a homeomorphism. 
For $M\in\N$ define a metric $d_M$ on $X$ by 
\begin{equation*}
    d_M(x,y)=\max\limits_{0\leq m<M}d(T^m x,T^m y)
\end{equation*}
and $B_{d_M}(x,r)=\{y\in X: d_M(x,y)\leq r\}$. We call $B_{d_M}(x,r)$ the $M$-th \textbf{Bowen ball} of radius $r$ centred at $x$.

The \textbf{topological entropy} of $(X,T)$ is defined by
\begin{equation*}
    \htop(X,T)=\lim_{\varepsilon\to0}\left(\lim_{M\to\infty}\frac{\log N_{d_M}(X,\varepsilon)}{M}\right).
\end{equation*}
The limit over $M$ exists since $\log N_{d_M}(X,\varepsilon)$ is sub-additive\footnote{A sequence $\N\ni n\mapsto a_n\in[0,\infty)$ is called \textbf{sub-additive} if $a_{m+n}\leq a_m+a_n$ for every $m,n\in\N$.
Fekete's lemma (\cite[Proposition 6.2.3]{Coo15}) states that if $n\mapsto a_n$ is a sub-additive sequence, then the limit $\lim_{n\to\infty}\frac{a_n}{n}$ exists and equals $\inf_{n\in\N}\frac{a_n}{n}$.} in $M$. 
Let $(X,T)$ and $(Y,S)$ be two \textit{TDSs} with a factor map $\pi:(X,T)\to(Y,S)$ and $d$ be a metric on $X$. The \textbf{topological conditional entropy} (\cite{Mis76conditional, Dow11}) of $\pi$ is defined by
\begin{equation*}
\begin{split}
    \htop(X,T\vert Y,S)&=\lim_{\varepsilon\to0}\left(\lim_{M\to\infty}\frac{\sup_{y\in Y}\log N_{d_M}(\pi^{-1}(y),\varepsilon)}{M}\right)\\
    &=\lim_{\varepsilon\to0}\left(\limsup_{M\to\infty}\frac{\sup_{y\in Y}\log \widehat{N}_{d_M}(\pi^{-1}(y),\varepsilon)}{M}\right).
\end{split}
\end{equation*}
Notice that the quantity $\sup_{y\in Y}\log N_{d_M}(\pi^{-1}(y),\varepsilon)$ is sub-additive in $M$.  The equality between the above expressions for $\htop(X,T\vert Y,S)$ follows from straightforward inequalities
\begin{equation}\label{eq: covering sep numbers ineq} 
N_d(K,\varepsilon)\leq \widehat{N}_d(K,\varepsilon / 4)\leq N_d(K,\varepsilon / 4),
\end{equation}
where $K$ is an arbitrary subset of $X$. In particular, it was shown in \cite{DS02fiber} that the topological conditional entropy coincides with the supremum of the topological \textit{fiber entropy}:
\begin{equation}\label{eq:fiber entropy}
    \htop(X,T\vert Y,S)=\sup_{y\in Y}\htop(\pi^{-1}(y),T).
\end{equation}
The topological entropy and topological conditional entropy are topological invariants of dynamical systems. In particular, they are independent of the choice of a compatible metric. 

A different measure of complexity is needed for systems with infinite topological entropy. A classical notion is the following one, introduced by Lindenstrauss and Weiss \cite{LW00}. The upper and lower \textbf{metric mean dimension} of the \textit{TDS} $(X,T,d)$ are defined as
    \begin{equation*}
        \ummdim(X,T,d)=\limsup\limits_{\varepsilon\to0}\lim\limits_{M\to\infty}\frac{\log N_{d_M}(X,\varepsilon)}{M\log(1/\varepsilon)}
    \end{equation*}
    and
    \begin{equation*}
        \lmmdim(X,T,d)=\liminf\limits_{\varepsilon\to0}\lim\limits_{M\to\infty}\frac{\log N_{d_M}(X,\varepsilon)}{M\log(1/\varepsilon)}.
    \end{equation*}
    If the upper and lower limits coincide, then we call its common value the \textbf{metric mean dimension} of $(X,T,d)$ and denote it by $\mmdim(X,T,d)$.

\subsection{Mean Assouad dimension and mean Assouad spectrum}\label{defn:mean Assouad dim}

Let us now define the mean Assouad dimension. It combines the ideas behind the Assouad dimension and metric mean dimension. First, we need the following proposition.
\begin{prop}\label{subadditivity}
    The sequence $\N\ni M\mapsto\sup\limits_{x\in X} \log N_{d_M}(B_{d_M}(x,r),\rho)$ is sub-additive for every $0<\rho<r$.
\end{prop}
\begin{proof}
    Fix $0<\rho<r$. Define $s_M(x)=N_{d_M}(B_{d_M}(x,r),\rho)$ for each $x\in X$ and $s_M=\sup\limits_{x\in X}s_M(x)$. Fix $M,N\in \N$. It suffices to prove that $s_{M+N}\leq s_M \cdot s_N$. For each $x\in X$, let $\{U_i(x)\}_{1\leq i\leq s_M(x)}$ be a cover of $B_{d_M}(x,r)$ with $\diam(U_i(x),d_M)\leq\rho$ for each $1\leq i\leq s_M(x)$, and let $\{V_j(x)\}_{1\leq j\leq s_N(x)}$ be a cover of $B_{d_N}(x,r)$ with $\diam(V_j(x),d_N)\leq\rho$ for each $1\leq j\leq s_N(x)$. Define $W(x)=\{U_i(x)\cap T^{-M}(V_j(T^Mx)): 1\leq i\leq s_M(x),1\leq j\leq s_N(T^Mx)\}=:\{W_{i,j}(x): 1\leq i\leq s_M(x),1\leq j\leq s_N(T^Mx)\}$. We now prove that $W(x)$ is a cover of $B_{d_{M+N}}(x,r)$ satisfying $\diam(W_{i,j}(x),d_{M+N})\leq\rho$ for each pair $(i,j)$. Indeed, if $y\in B_{d_{M+N}}(x,r)$, then $d_{M+N}(x,y)\leq r$ implies that $d_M(x,y)\leq r, d_N(T^Mx,T^My)\leq r$ and therefore $y\in B_{d_{M}}(x,r)\cap T^{-M}(B_{d_N}(T^Mx,r))$. Then there exist $1\leq i\leq s_M(x)$ and $1\leq j\leq s_N(T^Mx)$ such that $y\in U_i(x)\cap T^{-M}(V_j(T^Mx))=W_{i,j}(x)$. Moreover, for all $y,z\in W_{i,j}(x)$, we have
    \begin{equation*}
    \begin{split}
        d_{M+N}(y,z)&=\max\limits_{0\leq k<M+N}d(T^ky,T^kz)=\max\left\{\max\limits_{0\leq k<M}d(T^ky,T^kz),\max\limits_{0\leq k<N}d(T^{k+M}y,T^{k+M}z)\right\}\\
        &\leq\max\{\diam(U_i(x),d_M),\diam(V_j(T^Mx),d_N)\}\leq \rho.
    \end{split}
    \end{equation*}
    Thus, $s_{M+N}(x)=N_{d_{M+N}}(B_{d_{M+N}}(x,r),\rho)\leq \#W(x)\leq s_M(x)\cdot s_N(T^Mx)\leq s_M\cdot s_N$. Taking supremum in $x\in X$ on the LHS, we obtain $s_{M+N}\leq s_M \cdot s_N$ as desired.
\end{proof}

For $0<\rho<r$ we define
    \begin{equation*}
        S(X,r,\rho)=\lim\limits_{M\to\infty}\frac{1}{M}\sup\limits_{x\in X}\log N_{d_M}(B_{d_M}(x,r),\rho).
    \end{equation*}
The limit with respect to $M$ exists due to Proposition \ref{subadditivity}.
We define the \textbf{mean Assouad dimension} of the \textit{TDS} $(X,T,d)$ by
    \begin{equation}\label{eq:mean Assouad dim}
        \madim(X,T,d)=\inf\left\{s>0: \underset{C>0}{\exists}\ \underset{0<\rho<r}{\forall}\  e^{S(X,r,\rho)}\leq C(r/\rho)^s\right\}.
    \end{equation}

   \begin{rem}
        We always have
        \[ \madim(X,T,d) \leq \adim (X,d).\]
        Indeed, due to Proposition \ref{subadditivity}, for arbitrary pair of scales $0<\rho<r$, we have
        \begin{equation*}\label{eq: S infimum}
        S(X,r,\rho)=\inf\limits_{M\in\N}\frac{1}{M}\sup\limits_{x\in X} \log N_{d_M}(B_{d_M}(x,r),\rho)\leq\sup\limits_{x\in X} \log N_{d}(B_d(x,r),\rho)
        \end{equation*}
    and therefore
    \begin{equation*}
        \{s>0: \underset{C>0}{\exists}\  \underset{0<\rho<r}{\forall}\ \underset{x\in X}{\forall}\ N_{d}(B_d(x,r),\rho)\leq C(r/\rho)^s\}\subset \{s>0: \underset{C>0}{\exists}\ \underset{0<\rho<r}{\forall}\  e^{S(X,r,\rho)}\leq C(r/\rho)^s\}.
    \end{equation*}
    Taking infimum in $s>0$ leads to $\madim(X,T,d) \leq \adim (X,d)$.
   \end{rem}

For $\theta\in(0,1)$, we define the \textbf{mean Assouad spectrum} as
    \begin{equation}\label{eq: madim spectrum def}
        \madim^{\theta}(X,T,d)=\inf\left\{s>0: \underset{C>0}{\exists}\ \underset{0<r<1}{\forall}\  e^{S(X,r,r^{1/\theta})}\leq C(r/r^{1/\theta})^s\right\}.
    \end{equation}

    Here, the parameter $\theta$ fixes the relationship between $\rho$ and $r$ as $\rho=r^{1/\theta}$. Again, inequality $\madim^\theta(X,T,d) \leq \adim^\theta (X,d)$ holds.

If $(Y,T,d)$ is a subsystem of $(X,T,d)$, i.e. $Y$ is closed $T$-invariant subset of $X$, then $\madim(Y,T,d)\leq \madim(X,T,d)$ and $\madim^\theta(Y,T,d)\leq \madim^\theta(X,T,d)$. This follows directly from the fact that $S(X,r,\rho)\geq S(Y,r,\rho)$ for all $0<\rho<r$.

Note that the metric mean dimension, as well as the mean Assouad  dimension and spectra are all metric-dependent quantities and it is not difficult to see that these are not topological invariants.

\begin{rem}
    We can also define mean Assouad dimension and spectra using separated sets. More precisely, for $x\in X, M\in\N$ and $0<\rho<r$, we define $\widehat{N}_{d_M}(B_{d_M}(x,r),\rho)$ as the maximal cardinality of $\{y_1,\ldots,y_k\}\subset B_{d_M}(x,r)$ with $d_M(y_i,y_j)\geq \rho$ for $1\leq i\neq j\leq k$. Denote by $S_{\sep}(X,r,\rho) = \limsup \limits_{M \to \infty} \frac{1}{M}\sup_{x\in X}\log \widehat{N}_{d_M}(B_{d_M}(x,r),\rho)$. Then changing $S(X, r, \rho)$ to $S_{\sep}(X,r,\rho)$ in \eqref{eq:mean Assouad dim} and \eqref{eq: madim spectrum def} leads to the equivalent definitions. This follows from inequalities \eqref{eq: covering sep numbers ineq}.

       % Since $N_d(E,4\rho) \leq\widehat{N}_d(E,\rho)\leq N_d(E,\rho)$, we may replace the covering numbers in the definition of the Assouad dimension or mean Assouad dimension with packing numbers.
    \end{rem}

\begin{rem}
There are several other dynamical invariants that one can define by an analogy with fractal dimensions, in a manner similar to the mean Assouad dimension. For instance, it is natural to consider the \textbf{mean upper spectrum}, which fixes $r^{\frac{1}{\theta}}$ as an upper bound for $\rho$, defined by
    \begin{equation*}
        \overline{\mdim}_{\mathrm A}^{\theta}(X,T,d)=\inf\left\{s>0: \underset{C>0}{\exists}\ \underset{0<\rho\leq r^{1/\theta}<r<1}{\forall}\  e^{S(X,r,\rho)}\leq C(r/\rho)^s\right\} \text{ for } \theta \in (0,1).
    \end{equation*}
Similarly, one can define the \textbf{mean quasi-Assouad dimension} by
    \begin{equation*}
        \mqadim(X,T,d)=\lim\limits_{\theta\to1}\left(\inf\left\{s>0: \underset{C>0}{\exists}\ \underset{0<\rho\leq r^{1/\theta}<r<1}{\forall}\  e^{S(X,r,\rho)}\leq C(r/\rho)^s\right\}\right).
    \end{equation*}

It clearly holds that
\begin{equation*}
    \madim^{\theta}(X,T,d) \leq \overline{\mdim}_{\mathrm A}^{\theta}(X,T,d) \leq \mqadim (X,T,d) \leq \madim(X,T,d)
\end{equation*}
and by definition $\overline{\mdim}_{\mathrm A}^{\theta}(X,T,d)\to\mqadim (X,T,d)$ as $\theta\to1$.

Also, recall that lower dimension of a compact metric space is defined as
\begin{equation*}
        \ladim (X,d)=\sup\left\{s>0: \underset{C>0}{\exists}\ \underset{0<\rho<r\leq\diam(X)}{\forall}\ \underset{x\in X}{\forall}\  N_d(B_d(x,r),\rho)\geq C(r/\rho)^s\right\}
\end{equation*}
and the lower spectrum is defined as
\begin{equation*}
        \ladim^{\theta} (X,d)=\sup\left\{s>0: \underset{C>0}{\exists}\ \underset{0<r<1}{\forall}\ \underset{x\in X}{\forall}\  N_d(B_d(x,r),r^{1/\theta})\geq C(r/r^{1/\theta})^s\right\} \text{ for } \theta \in (0,1).
\end{equation*}
These are dual notions to the Assouad dimension and spectrum. We can define the \textbf{mean lower dimension} of the \textit{TDS} $(X,T,d)$ by
    \begin{equation*}
         \mldim(X,T,d)=\sup\left\{s>0: \underset{C>0}{\exists}\ \underset{0<\rho<r\leq\diam(X)}{\forall}\  e^{\widehat{S}(X,r,\rho)}\geq C(r/\rho)^s\right\},
     \end{equation*}
     where 
     \begin{equation*}
        \widehat{S}(X,r,\rho)=\liminf\limits_{M\to\infty}\frac{1}{M}\inf\limits_{x\in X}\log N_{d_M}(B_{d_M}(x,r),\rho)
    \end{equation*}
and the \textbf{mean lower spectrum} as
     \begin{equation*}
         \mldim^{\theta}(X,T,d)=\sup\left\{s>0: \underset{C>0}{\exists}\ \underset{0<r<1}{\forall}\  e^{\widehat{S}(X,r,r^{1/\theta})}\geq C(r/r^{1/\theta})^s\right\} \text{ for } \theta \in (0,1).
     \end{equation*}

We do not study these notions in detail in this work.

\end{rem}

\section{Basic properties}\label{results}

\subsection{Bi-Lipschitz invariance}

The Assouad dimension of a set is stable under bi-Lipschitz maps, see e.g. \cite[Lemma 2.4.2]{Fra21}. We prove that the mean Assouad dimension of \textit{TDS}s is stable under bi-Lipschitz isomorphisms.
\begin{prop}[{\bf Bi-Lipschitz invariance}{}]\label{bi-Lipschitz invariants}
     Mean Assouad dimension and spectrum are invariant under bi-Lipschitz isomorphisms. Precisely, if $(X,T,d)$ and $(Y,S,\tau)$ are \textit{TDS}s, $\Phi: X\to Y$ is bi-Lipschitz, i.e.
     \begin{equation}\label{bi-Lip}
         L^{-1}d(x,y)\leq \tau(\Phi(x),\Phi(y))\leq Ld(x,y),~\forall x,y\in X
     \end{equation}
     for some $L>1$, and satisfies $\Phi\circ T=S\circ \Phi$, then $\madim(X,T,d)=\madim(Y,S,\tau)$ and $\madim^\theta(X,T,d)=\madim^\theta(Y,S,\tau)$ for $\theta \in (0,1)$.
\end{prop}
\begin{proof}
    We will prove only the equality $\madim(X,T,d)=\madim(Y,S,\tau)$, as the proof of $\madim^\theta(X,T,d)=\madim^\theta(Y,S,\tau)$ is very similar.

     Let $s>\madim(X,T,d)$. Fix an arbitrary point $y\in Y$, arbitrary scales $0<\rho<r$ and $M\in\N$. By \eqref{bi-Lip}, we have $\Phi^{-1}(B_{\tau_M}(y,r))\subset B_{d_M}(\Phi^{-1}(y),Lr)$. 
     Indeed, if $x\in B_{\tau_M}(y,r)$, then $\tau_M(x,y)\leq r$ and therefore
     \begin{equation*}
     \begin{split}
         d_{M}(\Phi^{-1}(x),\Phi^{-1}(y))&=\max\limits_{0\leq k<M} d(T^k\circ\Phi^{-1}(x),T^k\circ\Phi^{-1}(y))=\max\limits_{0\leq k<M} d(\Phi^{-1}\circ S^kx,\Phi^{-1}\circ S^ky)\\
         &\underset{\eqref{bi-Lip}}{\leq}L\cdot \max\limits_{0\leq k<M} \tau(S^kx,S^ky)=L\cdot \tau_M(x,y)\leq Lr,
     \end{split}
     \end{equation*}
     which implies that $\Phi^{-1}(x)\in B_{d_M}(\Phi^{-1}(y),Lr)$.
     By the definition of $\madim(X,T,d)$, there exists a constant $C>0$ such that
     \begin{equation*}
         e^{S(X,Lr,L^{-1}\rho)}\leq C\left(\frac{Lr}{L^{-1}\rho}\right)^s=CL^{2s}\left(\frac{r}{\rho}\right)^s.
     \end{equation*}
     Now we show that if $\{U_i\}_{i\in\mathcal{I}}$ is a cover of $\Phi^{-1}(B_{\tau_M}(y,r))$ with $\diam(U_i,d_M)\leq L^{-1}\rho$ for each $i\in\mathcal{I}$, then $\{\Phi(U_i)\}_{i\in\mathcal{I}}$ is a cover of $B_{\tau_M}(y,r)$ with $\diam(\Phi(U_i),\tau_M)\leq \rho$ for each $i\in\mathcal{I}$. Indeed, fix $i\in\mathcal{I}$, for all $y_1,y_2$ in $\Phi(U_i)$ with $y_1=\Phi(x_1),y_2=\Phi(x_2)$, combining $\Phi\circ T=S\circ \Phi$ and \eqref{bi-Lip}, we have
     \begin{equation*}
         \begin{split}
             \tau_M(y_1,y_2)&=\max\limits_{0\leq k<M}\tau(S^ky_1,S^ky_2)=\max\limits_{0\leq k<M}\tau(S^k\circ\Phi(x_1),S^k\circ\Phi(x_2))\\
             &=\max\limits_{0\leq k<M}\tau(\Phi\circ T^kx_1,\Phi\circ T^kx_2)\leq L\cdot\max\limits_{0\leq k<M}d(T^kx_1,T^kx_2)\\
             &=Ld_M(x_1,x_2)\leq L\cdot L^{-1}\rho=\rho.
         \end{split}
     \end{equation*}
     This together with the fact $\Phi^{-1}(B_{\tau_M}(y,r))\subset B_{d_M}(\Phi^{-1}(y),Lr)$ imply that
     \begin{equation*}
         N_{\tau_M}(B_{\tau_M}(y,r),\rho)\leq N_{d_M}(\Phi^{-1}(B_{\tau_M}(y,r)),L^{-1}\rho)\leq N_{d_M}(B_{d_M}(\Phi^{-1}(y),Lr),L^{-1}\rho).
     \end{equation*}
     By the arbitrariness of $M\in \N$ and $y\in Y$, we have 
     \begin{equation*}
         S(Y,r,\rho)\leq S(X,Lr,L^{-1}\rho)
     \end{equation*}
     and therefore
     \begin{equation*}
         e^{S(Y,r,\rho)}\leq CL^{2s}(r/\rho)^s.
     \end{equation*}
     Since $s>\madim(X,T,d)$ is chosen arbitrarily, we conclude that $\madim(X,T,d)\geq\madim(Y,S,\tau)$. Finally, we may obtain the reverse inequality by using $\Phi^{-1}$ instead of $\Phi$ and repeating the above argument.
 \end{proof}

%\subsection{Basic properties of the mean Assouad and lower spectra}\label{Basic properties of dynamical spectra}
%In this section, we study some analytic properties of the dynamical spectra and obtain some non-trivial bounds on the dynamical spectra using the dynamical dimensions. We also investigate how the dynamical spectra distorts under bi-H\"{o}lder isomorphisms. Finally, we establish the product formulae.

\subsection{Bounds on the mean Assouad spectrum}

The following general bound for the mean Assouad spectrum is inspired by \cite[Proposition 3.1]{FY18}.
\begin{prop}\label{general bound}
    For all $0<\theta<1$, 
    \begin{equation*}
        \ummdim(X,T,d) \leq \madim^{\theta}(X,T,d)\leq\min\left\{\frac{\ummdim(X,T,d)}{1-\theta},\madim(X,T,d)\right\}.
    \end{equation*}
\end{prop}
\begin{proof}
%{\r [Adam: I have written a new version of the proof, so that now the cases $\mmdim (X,T,d) = \infty$ and $\madim^{\theta}(X,T,d) = \infty$ are included]}
Inequality $\madim^{\theta}(X,T,d)\leq\madim(X,T,d)$ for all $\theta\in(0,1)$ is obvious. 
    For proving $\madim^{\theta}(X,T,d) \leq \frac{\ummdim(X,T,d)}{1-\theta}$ we can assume $\ummdim(X,T,d) < \infty$. Then for $s>0$ such that $\ummdim(X,T,d)<s$ one has
    %\begin{equation}\label{eq: mmdim t s bound}
    %    t<\ummdim(X,T,d)<s.
    %\end{equation}
    that there exists $C_1>0$ such that for all $0 < r < 1$ and all $M$ large enough (depending on $r$), the following bound holds:
    \begin{equation*}%\label{lower bound}
       \sup\limits_{x\in X} N_{d_M}(B_{d_M}(x,r),r^{1/\theta}) \leq N_{d_M}(X,r^{1/\theta}) \leq C_1\left(\frac{1}{r^{1/\theta}}\right)^{sM} = C_1\left(\frac{r}{r^{1/\theta}}\right)^{\frac{s}{1-\theta}M}.
    \end{equation*}
    Therefore $\madim^{\theta}(X,T,d)\leq \frac{s}{1-\theta}$ and letting $s \searrow \ummdim(X,T,d)$ we obtain $\madim^{\theta}(X,T,d) \leq \frac{\ummdim(X,T,d)}{1-\theta}$.
    
    For the inequality $\ummdim(X,T,d) \leq \madim^{\theta}(X,T,d)$ we may assume $\madim^{\theta}(X,T,d) < \infty$. Take $s > 0$ such that $\madim^{\theta}(X,T,d) < s$. For $0<\eps<1$ note that we may cover $X$ by $M$-th Bowen balls of radius $1$ and then cover each of these balls by $M$-th Bowen balls of radius $\eps^{1/\theta}$. Iterating this $k$ times, we obtain for $k \geq 1$
    \begin{equation}\label{eq: theta madim prod bound}
        N_{d_M}(X,\eps^{k/\theta}) \leq N_{d_M}(X,1) \prod \limits_{i=0}^{k-1} \sup\limits_{x\in X}N_{d_M}(B_{d_M}(x,\eps^{i / \theta}),\eps^{ (i+1) /\theta}).
    \end{equation}
    By the definiton of $\madim^{\theta}(X,T,d)$, there exists a constant $C_2 > 0$ such that for every $\eps > 0$ and $k \in \N$ there is $M_0 = M_0(\eps,k)$ so that for every $M \geq M_0$ and $0 \leq i < k$
    \[ \sup\limits_{x\in X}N_{d_M}(B_{d_M}(x,\eps^{i / \theta}),\eps^{ (i+1) /\theta}) \leq C^M (\eps^{-1/\theta})^{sM}.  \]
    Applying this to \eqref{eq: theta madim prod bound} gives for $M \geq M_0$
    \begin{equation}\label{eq: mmdim theta bound}
        N_{d_M}(X,\eps^{k/\theta}) \leq N_{d_M}(X,1) C^{Mk}(\eps^{-k/\theta})^{sM}.
    \end{equation}
    Fix arbitrarily small $\delta > 0$ and fix $\eps>0$ small enough to guarantee $C\eps^{\delta / \theta} \leq 1$. Then \eqref{eq: mmdim theta bound} yields
    \[ N_{d_M}(X,\eps^{k/\theta}) \leq N_{d_M}(X,1) (C\eps^{\delta / \theta})^{Mk}(\eps^{-k/\theta})^{(s+\delta)M} \leq N_{d_M}(X,1) (\eps^{-k/\theta})^{(s+\delta)M}.  \]
    and hence for every $k \geq 1$
    \[ \limsup \limits_{M \to \infty} \frac{1}{M} \log N_{d_M}(X,\eps^{k/\theta}) \leq S(X) + (s+\delta) \log (\eps^{-k/\theta}),  \]
    where $S(X) := \limsup \limits_{M \to \infty} \frac{1}{M} \log N_{d_M}(X,1)$. As $X$ is compact, we have $S(X)<\infty$, so consequently
    \[  \ummdim(X,T,d) =  \limsup \limits_{k \to \infty} \limsup \limits_{M \to \infty} \frac{\log N_{d_M}(X,\eps^{k/\theta})}{M\log(\eps^{- k / \theta})} \leq s+\delta.  \]
    This proves $\ummdim(X,T,d) \leq \madim^{\theta}(X,T,d)$.

\end{proof}

 \begin{cor}
     $\madim^{\theta}(X,T,d)\to\ummdim(X,T,d)$ as $\theta\to0$.
 \end{cor}
\begin{proof}
    Letting $\theta\to0$ in Proposition \ref{general bound}, we obtain the desired result.
\end{proof}

\subsection{Characterization using stable sets}

For the metric mean dimension, Tsukamoto \cite[Theorem 1.1]{Tsu22} proved the following local characterization, in terms of $\delta$-stable sets. More precisely, he showed that for any $\delta>0$,
    \begin{equation*}
        \ummdim(X,T,d)=\limsup_{\varepsilon\to0}\sup_{x\in X} \limsup_{M\to\infty}\frac{\log N_{d_M}(B_{d_\Z}(x,\delta),\varepsilon)}{M\log(1/\varepsilon)}
    \end{equation*}
    and
    \begin{equation*}
        \lmmdim(X,T,d)=\liminf_{\varepsilon\to0}\sup_{x\in X} \limsup_{M\to\infty}\frac{\log N_{d_M}(B_{d_\Z}(x,\delta),\varepsilon)}{M\log(1/\varepsilon)}
    \end{equation*}
    where
    \begin{equation*}
        B_{d_{\Z}}(x,\delta)=\{y\in X: d(T^M x, T^M y)\leq\delta \text{ for all } M\in\Z \}.
    \end{equation*}

    The usefulness of this result lies in the fact that $B_{d_{\Z}}(x,\delta)$ is often a very small subset of $X$ - see for a more detailed discussion \cite{Tsu22}.
    
    We can prove a similar characterization of the mean Assouad dimension.

\begin{prop}\label{prop: madim stable def}
For $0 < \rho < r$ define

\[ S_{\Z}(X, r, \rho) = \sup \limits_{x \in X} \limsup \limits_{M \to \infty} \frac{1}{M} \log N_{d_M}(B_{d_{\Z}}(x,r),\rho) \]

Then
\[
\madim(X, T, d)  = \inf \left\{ s > 0 : \underset{C>0}{\exists}\ \underset{0 < \rho < r}{\forall}\ e^{S_{\Z}(X, r, \rho)} \leq C(r /\rho)^s \right\}.
\]
\end{prop}

\begin{proof}
As clearly $S_\Z(X, r, \rho) \leq  S(X, r, \rho)$, it suffices to establish the following inequality:

\begin{equation*}
S(X, r, \rho) \leq S_{\Z}(X, r, \rho/4).
\end{equation*}

It follows directly from \cite[Proposition 2.2]{Tsu22} (see also \cite[Proposition 2.2]{BowenExpansive72}).
\end{proof}

\subsection{Mean Assouad dimension of the non-wandering set}

A point $x \in X$ is called \textbf{non-wandering} if for every open neighbourhood $U$ of $x$ there exists $n \geq 1$ such that $T^n U \cap U \neq \emptyset$. Let $\Om(T)$ denote the set of all non-wandering points of the system. It is easy to see that $\Omega(T)$ is a $T$-invariant compact set. We can prove that the mean Assouad dimension of the system remains unchanged after restricting to the set of non-wandering points, similarly to other notions of complexity of a topological dynamical system (e.g. the topological entropy \cite{BowenTopA70, KSNonaut96}, mean topological dimension \cite[Lemma 7.2]{Gutman17}, metric mean dimension \cite[Theorem 4.6]{RANonaut22}).

\begin{prop}\label{prop: nonwandering}
Let $(X,T,d)$ be a topological dynamical system and let $\Omega(T)$ be its set of non-wandering points. Then
\[ \madim(X,T,d) = \madim(\Omega(T), T, d). \]
\end{prop}

For the proof we shall need some notation and preliminary lemmas. We shall denote $\Om := \Om(T)$, and for $M \in \N,\ 0 < \rho <r$ set
\[\Om_{M,\rho} = \{ x \in X : \dist_{d_M}(\Omega, x) < \rho \}\]
and
\[ K_\rho = \{ x \in X : \dist(\Om, x) \geq \rho\}.\]

\begin{lem}\label{lem: wandering visits}
For every $\rho > 0$ there exists $N = N(\rho) \in \N$ such that
\[ |\{ 0 \leq i < \infty : T^i x \in K_\rho\}| \leq N.\]
\end{lem}

\begin{proof}
For each $x \in K_\rho$, let $U_x$ be an open neighbourhood of $x$ such that $T^n U_x \cap U_x = \emptyset$ for every $n \geq 1$ (it exists as $x \notin \Om$). As $K_\rho$ is a compact set, there exists a finite subcover $\{ U_i : 1 \leq i \leq N\}$ of $\{ U_x : x \in K_\rho \}$. As every orbit can visit each set $U_i$ at most once and $K_\rho \subset \bigcup \limits_{i=1}^N U_i$, we see that every orbit can visit $K_\rho$ at most $N$ times. 
\end{proof}

\begin{lem}\label{lem: om cover bound}
The following inequality holds for all $x \in X,\ M \in \N$ and $0 < \rho < r$
\begin{equation}\label{eq: om cover bound}
N_{d_M}(B_{d_M}(x,r) \cap \Omega_{M,\rho}, 3\rho) \leq \sup \limits_{y \in \Omega(T)} N_{d_M}(B_{d_M}(y, 4r)\cap \Omega(T), \rho).
\end{equation}
\end{lem}

\begin{proof}
Denote $\Om := \Om(T)$. If $B_{d_M}(x,r) \cap \Omega_{M,\rho} = \emptyset$, then \eqref{eq: om cover bound} holds trivially. Otherwise, take $z_0 \in B_{d_M}(x,r) \cap \Omega_{M,\rho} $ and $y_0 \in \Om$ such that $d_M(z_0, y_0) < \rho$. Take a cover
\[ B_{d_M}(y_0, 4r)\cap \Omega \subset \bigcup \limits_{i=1}^N U_i, \]
where $N = N_{d_M}(B_{d_M}(y_0, 4r)\cap \Omega(T), \rho)$ and $\diam(U_i,d_M) < \rho$. Set $V_i = \{ u \in X : \dist_{d_M}(U_i, u) < \rho\}$ and note that $\diam(V_i,d_M) < 3\rho$. We claim that
\begin{equation}\label{eq: V cover}
B_{d_M}(x,r) \cap \Omega_{M,\rho} \subset \bigcup \limits_{i=1}^N V_i.
\end{equation}
Proving this will yield \eqref{eq: om cover bound}. Take $z \in B_{d_M}(x,r) \cap \Omega_{M,\rho}$. There exists $y \in \Om$ such that $d_M(y,z) < \rho$. Then $d_M(y,y_0) \leq d_M(y, z) + d_M(z,x) + d_M(x, z_0) + d_M(z_0, y_0) < 4r$, so $y \in B_{d_M}(y_0, 4r)$. Therefore there exists $1 \leq i \leq N$ such that $y \in U_i$ and hence $z \in V_i$. This proves \eqref{eq: V cover}.
\end{proof}

\begin{proof}[{Proof of Proposition \ref{prop: nonwandering}}]
For $0 < \rho < r$ and $M \in \N$ set
\[ s_M(\Om, r, \rho) = \sup \limits_{x \in \Om} N_{d_M}(B_{d_M}(x,r) \cap \Om, \rho).\]
Note that for the restriction of $T$ to $\Om$ one has
\[ S(\Om, r, \rho) = \lim \limits_{M \to \infty} \frac{1}{M} \log s_M(\Om, r, \rho).\]
Therefore, for fixed $s > \madim(\Om, T, d)$ there exists $C > 0$ such that
\begin{equation}\label{eq: s bound}
e^{S(\Om, r,\rho)} \leq C (r/\rho)^s \text{ for all } 0 < r < \rho.
\end{equation} 
Let us fix such $s$. Fix $0 < \rho < r$ and $\eps > 0$, There exists $M_0 \in \N$ such that 
\begin{equation}\label{eq: sM limit bound}
\frac{1}{M} \log s_M(\Om, r, \rho) \leq S(\Om, r, \rho) + \eps \text{ for all } M \geq M_0.
\end{equation} 
Let $N = N(\rho)$ be an in Lemma \ref{lem: wandering visits}. Fix now $M \in \N$ with $M\geq M_0(N+1)$. Let
\[ \Lambda = \{ I \subset \{0, 1, \ldots, M-1\} : |I| \leq N \} \]
and for $x \in X$ set
\[ I_x = \{ 0 \leq i < M : T^i x \in K_\rho \}. \]
By Lemma \ref{lem: wandering visits}, $I_x \in \Lambda$ for every $x \in X$. Therefore, setting $ X_I = \{ x \in X : I_x = I \}$ we have $X = \bigcup \limits_{I \in \Lambda} X_I$ and hence
\begin{equation}\label{eq: Lambda sum bound}
N_{d_M}(B_{d_M}(x,r), \rho) \leq \sum \limits_{ I \in \Lambda} N_{d_M}(B_{d_M}(x,r) \cap X_I, \rho).
\end{equation} 
Let us now fix $I \in \Lambda$ and partition $\{ 0, 1, \ldots, M-1\}$ as follows:
\begin{equation}\label{eq: I partition} \{ 0, 1, \ldots, M-1\} = \bigcup \limits_{i \in I} \{i\} \cup \bigcup \limits_{J \in \mP_I} J,
\end{equation}
where $\mP_I$ is the collection of maximal subintervals of $\{ 0, 1, \ldots, M-1\}$ which are disjoint with $I$.  In other words, $\mP_I$ is the collection of all sets $J \subset \{ 0, 1, \ldots, M-1\}$ of the form $J = \{ 0 \leq i < M : j_0 \leq i \leq j_1  \}$ with $0 \leq j_0 \leq j_1 < M,\ J \cap I = \emptyset$ satisfying the conditions
\[ j_0 = 0 \text{ or } j_0 - 1 \in I \]
and
\[ j_1 = M-1 \text{ or } j_1 + 1 \in I. \]
Note that $|\mP_I| \leq N+1$ and so
\begin{equation}\label{eq: J leq M0}
 \sum \limits_{\substack{J \in \mP_I\\ |J| < M_0}} \# J \leq M_0(N+1).     
\end{equation}
We shall now make use of the following two observations:
\[N_{d_{M_1 + M_2}}(A, \rho) \leq N_{d_{M_1}}(A, \rho)N_{d_{M_2}}(T^{M_1}A, \rho) \text{ for } A \subset X \text{ and } M_1, M_2 \in \N\]
and for $J \in \mP_I$
\[ T^{\min J} (B_{d_M}(x,r) \cap X_I) \subset B_{d_{M - \min J}}(T^{\min J}x,r) \cap \Om_{|J|, \rho} \subset B_{d_{|J|}}(T^{\min J}x,r) \cap \Om_{|J|, \rho}. \]
Using those repeatedly over the partition in \eqref{eq: I partition} together with \eqref{eq: J leq M0}, we obtain for $x \in X$
\begin{equation}\label{eq: fixed I bound}
\begin{split}
N_{d_M}(B_{d_M}(x,r) \cap X_I, \rho) & \leq N_{d_1}(X,\rho)^{|I|} \prod \limits_{J \in \mP_I}N_{d_{|J|}}(B_{d_{|J|}} (T^{\min J} x, r) \cap \Om_{|J|, \rho}, \rho) \\
& \leq N_{d_1}(X,\rho)^{N+M_0(N+1)} \prod \limits_{\substack{J \in \mP_I\\ |J| \geq M_0}}N_{d_{|J|}}(B_{d_{|J|}} (T^{\min J} x, r) \cap \Om_{|J|, \rho}, \rho).
\end{split}
\end{equation}
Lemma \ref{lem: om cover bound} and \eqref{eq: sM limit bound} give for $J \in \mP_I$ such that $|J| \geq M_0$
\[
\begin{split}
N_{d_{|J|}}(B_{d_{|J|}} (T^{\min J} x, r) \cap \Om_{|J|, \rho}, \rho) & \leq \sup \limits_{y \in \Om} N_{d_{|J|}}(B_{d_{|J|}} (y , 4r) \cap \Om, \rho / 3) \\
& = s_{|J|}(\Om, 4r, \rho / 3) \\
& \leq e^{|J|(S(\Om, 4r, \rho / 3) + \eps)}.
\end{split}
\]
Combining this with \eqref{eq: fixed I bound} and setting $C(r,\rho) = N_{d_1}(X,\rho)^{N+M_0(N+1)}$ we obtain
\[
\begin{split}
N_{d_M}(B_{d_M}(x,r) \cap X_I, \rho) & \leq C(r, \rho) \exp \left( (S(\Om, 4r, \rho / 3) + \eps) \sum \limits_{\substack{J \in \mP_I\\ |J| \geq M_0}} |J| \right) \\
& \leq C(r, \rho) e^{(S(\Om, 4r, \rho / 3) + \eps)M}.
\end{split}
\]
Applying this to \eqref{eq: Lambda sum bound} and noting that $|\Lambda| \leq M^N$ gives
\[ N_{d_M}(B_{d_M}(x,r), \rho) \leq M^N C(r,\rho)e^{(S(\Om, 4r, \rho / 3) + \eps)M} \]
and hence
\[ \frac{1}{M} \sup \limits_{x \in X} \log N_{d_M}(B_{d_M}(x,r) \leq \frac{N}{M} \log M + \frac{\log C(r,\rho)}{M} + S(\Om, 4r, \rho / 3) + \eps.  \]
Letting $M \to \infty$ and taking arbitrarily small $\eps > 0$ yields
\[ S(X, r, \rho) \leq S(\Om, 4r, \rho / 3).  \]
By \eqref{eq: s bound}
\[ e^{S(X, r, \rho)} \leq e^{S(\Om, 4r, \rho / 3)} \leq 12^s C(r / \rho)^s. \]
As we can choose $s$ arbitrarily close to $\madim(\Om, T, d)$, we conclude that $\madim(X,T,d) \leq \madim(\Om, T, d)$. As the opposite inequality is trivial, this finishes the proof.
\end{proof}

\section{Examples}\label{sec: examples}

In this section we present calculations of the mean Assouad dimesnion and spectrum in two basic cases - full shifts and band-limited functions.

\subsection{Full shifts}

Let $(F,d)$ be a compact metric space. Let $(F^{\N},\sigma)$ be the one-sided full-shift on the alphabet $F$ equipped with the standard product metric
\begin{equation*}
    \dd((x_n)_{n\in\N}, (y_n)_{n\in\N})=\sum_{n\in\N}2^{-n}d(x_n,y_n).
\end{equation*}
It is interesting to note \cite[Theorem 5]{VV17} and \cite[Corollary 16]{Shi22} proved that $\ummdim(F^{\N},\sigma,\dd)=\udim (F,d)$ and $\lmmdim(F^{\N},\sigma,\dd)=\ldim (F,d)$ for the full-shifts on any compact alphabets. Their proofs are based on variational principles for the metric mean dimension. The purpose of this section is to prove that the mean Assouad dimension of the full-shift on any compact alphabet coincides with the Assouad dimension of the alphabet without using the variational principle.
More precisely, we show that:
\begin{prop}\label{madim=adim}
    It holds $\madim(F^{\N},\sigma,\dd)=\adim(F,d)$ and $\madim^{\theta}(F^{\N},\sigma,\dd)=\adim^{\theta}(F,d)$ for all $\theta\in (0,1)$.
\end{prop}

For $N\in\N$, we define projections $\pi_N: F^{\N}\to F^N$ as $\pi_N(x_1,x_2,\ldots)=(x_1,\ldots,x_N)$.
For $0<\rho<r$, we denote 
        \begin{equation*}
    S_{\infty}(F^{\N},r,\rho)=\lim_{N\to\infty}\frac{1}{N}\sup_{x\in F^{\N}} \log N_{d_{\infty}^{N}}(B_{d_{\infty}^{N}}(\pi_N(x),r),\rho),
\end{equation*}
where 
\begin{equation*}
    B_{d_{\infty}^{N}}(\pi_N(x),r)=\{\xi\in F^N: d_{\infty}^{N}(\pi_N(x),\xi)\leq r\}
\end{equation*}
is the ball in $F^N$ centred at $\pi_N(x)$ and of radius $r$ and $d_{\infty}^{N}$ is the $\ell^\infty$-distance on $F^N$. The limit with respect to $N$ exists since the quantity $\sup_{x\in F^{\N}} \log N_{d_{\infty}^{N}}(B_{d_{\infty}^{N}}(\pi_N(x),r),\rho)$ is sub-additive in $N$ for arbitrary $0<\rho<r$.

The following lemma is inspired by \cite[Proposition A.3 \& Lemma A.4]{mmdimcompress}. The key point is one can define the mean Assouad dimension and spectrum of the full-shift equivalently in a way which avoids the direct use of both the shift map $\sigma$ and the metric $\dd$. Instead, one uses only the infinity norm on $F^N$ and the projection $\pi_N: F^{\N}\to F^N$.
    \begin{lem}\label{equivalent defn}
        If we further define
\begin{equation*}
    \madim(F^{\N})=\inf\{s>0: \underset{C>0}{\exists}\ \underset{0<\rho<r}{\forall}\  e^{S_{\infty}(F^{\N},r,\rho)}\leq C(r/\rho)^s\}
\end{equation*}
and 
\begin{equation*}
    \madim^{\theta}(F^{\N})=\inf\{s>0: \underset{C>0}{\exists}\ \underset{0<r<1}{\forall}\  e^{S_{\infty}(F^{\N},r,r^{1/\theta})}\leq C(r/r^{1/\theta})^s\}
\end{equation*}
for $\theta\in(0,1)$. Then we have $\madim(F^{\N},\sigma,\dd)=\madim(F^{\N})$ and $\madim^{\theta}(F^{\N},\sigma,\dd)=\madim^{\theta}(F^{\N})$ for all $\theta\in(0,1)$.
    \end{lem}
    \begin{proof}
        First, we prove that $\madim(F^{\N},\sigma,\dd)\leq \madim(F^{\N})$. Fix $x\in F^{\N}, 0<\rho<r$ and $N\in\N$. We immediately have $\pi_N(B_{\dd_N}(x,r))\subset B_{d_{\infty}^N}(\pi_N(x),r)$ since
        \begin{equation*}
            d_{\infty}^N(\pi_N(y),\pi_N(z))\leq \dd_N(y,z)
        \end{equation*}
        for all $y,z\in F^{\N}$. Let $L\in\N$ satisfy that $\sum_{n>L}2^{-n}\cdot \diam(F,d)<\rho/2$. Then
        \begin{equation}\label{Bowen metric and infinity metric}
            \dd_N(y,z)<d_{\infty}^{N+L}(\pi_{N+L}(y),\pi_{N+L}(z))+\frac{\rho}{2}
        \end{equation}
        and thus
        \begin{equation*}
            N_{\dd_N}(B_{\dd_N}(x,r),\rho)\leq N_{d_{\infty}^{N+L}}\left(\pi_{N+L}(B_{\dd_N}(x,r)),\frac{\rho}{2}\right).
        \end{equation*}
        Observe that
        \begin{equation*}
        \begin{split}
            \pi_{N+L}(B_{\dd_N}(x,r))&=\pi_N(B_{\dd_N}(x,r)) \times \pi_L(\sigma^N(B_{\dd_N}(x,r)))\\
            &\subset \pi_N(B_{\dd_N}(x,r)) \times F^L\subset B_{d_{\infty}^N}(\pi_N(x),r)\times F^L
        \end{split}
        \end{equation*}
        and therefore
        \begin{equation*}
            N_{d_{\infty}^{N+L}}\left(\pi_{N+L}(B_{\dd_N}(x,r)),\frac{\rho}{2}\right) \leq N_{d_{\infty}^N}\left(B_{d_{\infty}^N}(\pi_N(x),r),\frac{\rho}{2}\right) \cdot N_{d_{\infty}^L}\left(F^L,\frac{\rho}{2}\right).
        \end{equation*}
        This yields that
        \begin{equation*}
            \begin{split}
                &\lim_{N\to\infty}\frac{\sup\limits_{x\in F^{\N}}\log N_{\dd_N}(B_{\dd_N}(x,r),\rho)}{N}\\
                &\leq \lim_{N\to\infty}\frac{\sup\limits_{x\in F^{\N}}\log N_{d_{\infty}^{N}}\left(B_{d_{\infty}^{N}}(\pi_N(x),r),\frac{\rho}{2}\right)+\log N_{d_{\infty}^{L}}\left(F^L,\frac{\rho}{2}\right)}{N}\\
                &=\lim_{N\to\infty}\frac{\sup\limits_{x\in F^{\N}}\log N_{d_{\infty}^{N}}\left(B_{d_{\infty}^{N}}(\pi_N(x),r),\frac{\rho}{2}\right)}{N}.
            \end{split}
        \end{equation*}
        That is to say, $S(F^{\N},r,\rho)\leq S_{\infty}(F^{\N},r,\frac{\rho}{2})$ holds for arbitrary $0<\rho<r$. Therefore $\madim(F^{\N}, \sigma, \dd)\leq\madim(F^{\N})$. 
    
    We now prove the reverse inequality. Note $B_{d_{\infty}^{N+L}}(\pi_{N+L}(x),r)\subset \pi_{N+L}(B_{\dd_N}(x,\frac{3}{2}r))$\footnote{Here one uses inequality \eqref{Bowen metric and infinity metric}.}. This yields that
    \begin{equation*}
        \begin{split}
            &\lim_{N\to\infty}\frac{\sup\limits_{x\in F^{\N}}\log N_{d_{\infty}^{N+L}}(B_{d_{\infty}^{N+L}}(\pi_{N+L}(x),r),\rho)}{N}\\
            &\leq \lim_{N\to\infty}\frac{\sup\limits_{x\in F^{\N}}\log N_{d_{\infty}^{N+L}}(\pi_{N+L}(B_{\dd_N}(x,\frac{3}{2}r)),\rho)}{N}\\
            &\leq \lim_{N\to\infty}\frac{\sup\limits_{x\in F^{\N}}\log N_{d_{\infty}^{N}}(\pi_{N}(B_{\dd_N}(x,\frac{3}{2}r)),\rho)+\log N_{d_{\infty}^{L}}(F^L,\rho)}{N}.\\
        \end{split}
    \end{equation*}
    Observe also that if $\{U_i(x)\}_{i=1}^{K(x)}$ is a cover of $B_{\dd_N}(x,\frac{3}{2}r)$ with $\diam(U_i(x),\dd_N)<\rho$ for each $1\leq i\leq K(x)$, then $\{\pi_N(U_i(x))\}_{i=1}^{K(x)}$ is a cover of $\pi_N(B_{\dd_N}(x,\frac{3}{2}r))$ with $\diam(\pi_N(U_i(x)),d_{\infty}^{N})<\rho$. Hence
    \begin{equation*}
       N_{d_{\infty}^{N}}(\pi_{N}(B_{\dd_N}(x,\frac{3}{2}r)),\rho)\leq
       N_{\dd_N}(B_{\dd_N}(x,\frac{3}{2}r),\rho).
    \end{equation*}
    We further obtain that
    \begin{equation*}
        \lim_{N\to\infty}\frac{\sup\limits_{x\in F^{\N}}\log N_{d_{\infty}^{N+L}}(B_{d_{\infty}^{N+L}}(\pi_{N+L}(x),r),\rho)}{N+L}
        \leq \lim_{N\to\infty}\frac{\sup\limits_{x\in F^{\N}}\log N_{\dd_N}(B_{\dd_N}(x,\frac{3}{2}r),\rho)}{N}
    \end{equation*}
    and therefore $S_{\infty}(F^{\N},r,\rho)\leq S(F^{\N},\frac{3}{2}r,\rho)$. By the arbitrariness of $\rho$ and $r$, we finally obtain that $\madim(F^{\N}, \sigma, \dd)\geq\madim(F^{\N})$. Up to now, we have proven that $\madim(F^{\N}, \sigma, \dd)=\madim(F^{\N})$. Similarly, we may also obtain that $\madim^{\theta}(F^{\N}, \sigma, \dd)=\madim^{\theta}(F^{\N})$ for all $\theta\in(0,1)$. 
    \end{proof}
    % We omit the proof of Claim \ref{equivalent defn} since it is essentially the same to that of Lemma \ref{lem:equivalent definition using projection and infinity norm}.

\begin{proof}[Proof of Proposition \ref{madim=adim}]
   First, we prove that $\madim(F^{\N},\sigma,\dd)\leq \adim (F,d)$. Let $s>\adim (F,d)$. Then there exists $C_0>0$ such that 
   \begin{equation*}
       \sup_{x\in F}N_d(B_d(x,r),\rho)\leq C_0\left(\frac{r}{\rho}\right)^s
   \end{equation*}
   for all $0<\rho<r$, where $B_d(x,r)=\{y\in F: d(x,y)\leq r\}$. Fix $0<\rho<r$. Note the quantity $N\mapsto \sup_{x\in F^{\N}} \log N_{d_{\infty}^{N}}(B_{d_{\infty}^{N}}(\pi_N(x),r),\rho)$ is sub-additive. Then 
   \begin{equation*}
   \begin{split}
       S_{\infty}(F^{\N},r,\rho)&=\inf_{N\in\N}\frac{1}{N}\sup_{x\in F^{\N}} \log N_{d_{\infty}^{N}}(B_{d_{\infty}^{N}}(\pi_N(x),r),\rho)\\
       &\leq \sup_{x\in F^{\N}} \log N_{d}(B_{d}(\pi_1(x),r),\rho)=\sup_{x\in F}\log N_d(B_d(x,r),\rho).
   \end{split}
   \end{equation*}
   This implies that $e^{S_{\infty}(F^{\N},r,\rho)}\leq C_0(r/\rho)^s$, and thus $\madim(F^{\N},\sigma,\dd)\leq s$ since $\rho$ and $r$ are arbitrary. Letting $s\to \adim (F,d)$, we conclude that $\madim(F^{\N},\sigma,\dd)\leq \adim (F,d)$.
   
   We now prove the opposite direction. 
   For $E\subset F$ and $\rho>0$, recall $\widehat{N}_d(E,\rho)$ denotes the largest possible cardinality of a $\rho$-packing by closed balls of radius $\rho$, centred at $E$. 
   % Since $N_d(E,4\rho) \leq\widehat{N}_d(E,\rho)\leq N_d(E,\rho)$, we may replace the covering numbers in the definition of the Assouad dimension or mean Assouad dimension with packing numbers.
   Let $t>0$ with $t<\adim(F,d)$. Fix $C>0$ and $n\in\N$. We may find a point $x\in F$ and small enough $0<\rho<r$ such that
    \begin{equation*}
        \widehat{N}_d(B_d(x,r),\rho)>C\left(\frac{r}{\rho}\right)^t.
    \end{equation*}
    Let $\{B_d(y_k,\rho):1\leq k\leq \widehat{N}_d(B_d(x,r),\rho)\}$ be the optimal $\rho$-packing of $B_d(x,r)$. We consider the point $(x,\ldots,x)\in F^n$.
    Note that $B_{d_{\infty}^n}((x,\ldots,x),r)=B_d(x,r)\times\cdots\times B_d(x,r)$.
    It is direct that $\{B_{d_{\infty}^n}((y_{j_1},\ldots,y_{j_n}),\rho):1\leq j_k\leq \widehat{N}_d(B_d(x,r),\rho),1\leq k\leq n\}$ is a $\rho$-packing of $B_{d_{\infty}^n}((x,\ldots,x),r)$. Then we obtain that
    \begin{equation*}
        N_{d_{\infty}^n}(B_{d_{\infty}^n}((x,\ldots,x),r),\rho)\geq
        \widehat{N}_{d_{\infty}^n}(B_{d_{\infty}^n}((x,\ldots,x),r),\rho)\geq  \left(\widehat{N}_d(B_d(x,r),\rho)\right)^n
        \geq  C^n\left(\frac{r}{\rho}\right)^{nt}
    \end{equation*}
    and therefore
    \begin{equation*}
        S_{\infty}(F^{\N},r,\rho)\geq \log C+t\log\frac{r}{\rho}.
    \end{equation*}
    This yields that for every $C>0$, there exist $0<\rho<r$ such that $e^{S_{\infty}(F^{\N},r,\rho)}\geq C(r/\rho)^t$, and thus $\madim(F^{\N},\sigma,\dd)\geq t$. Letting $t\to\adim (F,d)$, we finally obtain that $\madim(F^{\N},\sigma,\dd)\geq \adim(F,d)$, which shows that $\madim(F^{\N},\sigma,\dd)=\adim(F,d)$. Similarly, we may also obtain that $\madim^{\theta}(F^{\N},\sigma,\dd)=\adim^{\theta}(F,d)$ for all $\theta\in (0,1)$.
\end{proof}

The following example shows that the general upper bound in terms of the mean Assouad dimension in Proposition \ref{general bound} can be sharp.

\begin{example}
    Fix $\lambda>0$ and let $F_{\lambda}=\{0\}\cup\left\{\frac{1}{n^{\lambda}} \vert n\in\N\right\}=\left\{0,1,\frac{1}{2^{\lambda}},\frac{1}{3^{\lambda}},\ldots\right\}$. We define a metric $\dd$ on the  full-shift $F_{\lambda}^{\N}$ by
    \begin{equation*}
        \dd((x_n)_{n\in\N},(y_n)_{n\in\N})=\sum_{n\in\N} 2^{-n}|x_n-y_n|.
    \end{equation*}
    Let $\sigma: F_{\lambda}^{\N}\to F_{\lambda}^{\N}$ be the (left) shift map.
    Then $\mmdim(F_{\lambda}^{\N},\sigma,\dd)=\bdim F_{\lambda}=\frac{1}{1+\lambda}$ and $\madim(F_{\lambda}^{\N},\sigma,\dd)=\adim F_{\lambda}=1$. Moreover, combining Proposition \ref{madim=adim} and \cite[Corollary 6.4]{FY18}, for all $\theta\in(0,1)$ we have
    \begin{equation*}
    \begin{split}
        \madim^{\theta}(F_{\lambda}^{\N},\sigma,\dd)&=\adim^{\theta} F_{\lambda}\\
        &=\min\left\{\frac{1}{(1+\lambda)(1-\theta)}, 1\right\}\\
        &=\min\left\{\frac{\mmdim(F_{\lambda}^{\N},\sigma,\dd)}{1-\theta}, \madim(F_{\lambda}^{\N},\sigma,\dd)\right\}.
    \end{split}
    \end{equation*}
\end{example}

\begin{rem}In \cite[Theorem A]{Rut24}, Rutar investigated all possible forms of the Assouad spectra of sets in Euclidean spaces. More precisely, if a function $\phi: (0,1)\to [0,d]$ satisfies
    \begin{equation*}\label{attainable forms of Assouad spectra}
        0\leq (1-\theta_1)\phi(\theta_1)-(1-\theta_2)\phi(\theta_2)\leq (\theta_2-\theta_1)\phi\left(\frac{\theta_1}{\theta_2}\right)
    \end{equation*}
    for all $0<\theta_1<\theta_2<1$, there exists a set $ F\subset \R^d$ for which $\adim^{\theta} F=\phi(\theta)$ for all $\theta\in (0,1)$. Based on Rutar's classification result and Proposition \ref{madim=adim}, the mean Assouad spectra of the full-shift on an alphabet is fully understood, if the alphabet is a compact subset of $\R^d$. Recently, Orgov\'{a}nyi and Rutar \cite{OR25branching} considered a more refined problem of studying attainable forms of the two-scale branching function (in the non-dynamical setting), defined by
    \begin{equation*}
        \beta_X(u,v)=\log \sup_{x\in X}N_d(B_d(x,e^{-v}),e^{-u}),
    \end{equation*}
    where $0\leq v\leq u$. They have obtained classification results for $\beta_X$. Studying branching functions is essential for obtaining full dimension results in certain settings, e.g. for the lower box-counting dimension for attractors of infinite iterated function systems \cite{BanajiRutar24} (see also \cite{RutarCFNotes26}). It is easy to notice that for a \textit{TDS} $(X,T, d)$, function $\gamma_X(u,v) = S(X, e^{-v}, e^{-u})$ is a two-scale branching function (in the sense of \cite[Definition 1.1]{OR25branching}) and it is $\alpha$-Lipschitz up to an $o(u)$ error for $\alpha = \mqadim(X,T,d)$ (by the same arguments as in \cite[Lemma 2.4]{OR25branching}). This opens the possibility of using the branching functions machinery for studying complexity of dynamical systems.
\end{rem}

\subsection{Band-limited functions}

We can compute the mean Assouad dimension and mean Assouad spectrum of the space of band-limited functions with respect to the shift transformation. Its significance comes from the fact that this class of systems played a crucial role in the proof of the celebrated embedding theorem of Gutman and Tsukamoto for mean topological dimension \cite{GT20}. Let us recall its definition from \cite{GT20}. Given a bounded continuous function $\varphi : \R \to \C$, we consider its Fourier transform $\mF(\varphi)$ understood as a tempered distribution, see e.g. \cite[Chapter 7]{RudinFunctional}. Given $a<b$, we denote
\[ V[a,b] = \{ \varphi: \R \to \C\ |\ \supp \mF(\varphi) \subset [a,b] \}.
\]
Equipped with the $L^\infty(\R)$-norm, $V[a,b]$ is a Banach space. We consider further the unit ball space
\[ B_1(V[a,b]) = \{ \varphi \in V[a,b] : \|\varphi\|_{L^{\infty}} \leq 1 \} \]
and endow it with the metric
\[ d(\varphi_1, \varphi_2) = \sum \limits_{n=1}^\infty 2^{-n}\|\varphi_1 - \varphi_2 \|_{L^{\infty}([-n,n])}, \]
which renders $B_1(V[a,b])$ a compact metric space (see \cite[Lemma 2.3]{GT20}). Note that $d$ is translation-invariant, i.e. $d(\varphi_1 + \psi, \varphi_2 + \psi) = d(\varphi_1, \varphi_2)$. Finally, we consider the homeomorphism of $B_1(V[a,b])$ given by the shift map
\[\sigma : B_1(V[a,b]) \to B_1(V[a,b]),\ \sigma(\varphi)(t) = \varphi(t+1),\]
so that $(B_1(V[a,b]), \sigma, d)$ is a topological dynamical system.

It turns out that the above system is highly homogeneous, in the sense that its mean Assouad dimension, mean Assouad spectrum and metric mean dimension coincide:

\begin{prop} For every $\theta \in (0,1)$
\[\mmdim(B_1(V[a,b]), \sigma, d) = \madim^{\theta}(B_1(V[a,b]), \sigma, d) = \madim(B_1(V[a,b]), \sigma, d) = 2(b-a). \]
\end{prop}

\begin{proof}

It is well-known that $\mmdim(B_1(V[a,b]), \sigma, d) = 2(b-a)$ - see for instance \cite[footnote 4.]{GQT19} for the sketch of the argument in the case of the mean topological dimension. For the metric mean dimension one has to obtain the same upper bound, and it follows by noting that the embedding used in the argument - given via the sampling theorem - is in fact Lipschitz continuous. Therefore, due to Proposition \ref{general bound}, it suffices to prove
\begin{equation}\label{eq: bd lim madim leq mmdim}
\madim(B_1(V[a,b]), \sigma, d) \leq \mmdim(B_1(V[a,b]), \sigma, d).
\end{equation}

For that we shall apply Proposition \ref{prop: madim stable def}, and give an upper bound on $N_{d_M}(B_{d_\Z}(x,r),\rho)$ for $0< \rho < r$, where 
\[B_{d_{\Z}}(\varphi,r)=\{\psi\in B_1(V[a,b]): d(\sigma^M \varphi, \sigma^M \psi)\leq r \text{ for all } M\in\Z \}.\]
Note that $B_{d_{\Z}}(\varphi,r)$ is $\sigma$-invariant and in our case one simply has
\begin{equation}\label{eq: stable set V}
B_{d_{\Z}}(\varphi,r) = \{ \psi \in B_1(V[a,b]) : \|\psi - \varphi\|_{L^\infty(\R)} \leq r  \}.
\end{equation}
Fix $\eps>0$. Then there exists $\delta_0 > 0$ so that for all $0 < \delta \leq \delta_0$ there is $M_0(\delta) \in \N$ such that for all $M \geq M_0(\delta)$
\begin{equation}\label{eq: mmdim upper bound}
N_{d_M}(B_1(V[a,b]), \delta) \leq (1 / \delta)^{M(\mmdim(B_1(V[a,b]), \sigma, d) + \eps)}.
\end{equation} 
Fix $\varphi \in B_1(V[a,b])$ and $r>0.$ Consider the following map:
\[ F: B_{d_{\Z}}(\varphi,r) \to B_1(V[a,b]),\ F(\psi) = \frac{\psi - \varphi}{r}.  \]
Note that indeed $F(B_{d_{\Z}}(\varphi,r)) \subset  B_1(V[a,b])$ due to \eqref{eq: stable set V} and the linearity of the Fourier transform $\mF$. Moreover, $F$ satisfies
\[ F(\sigma(\psi)) = \sigma(F(\psi)) + \frac{\sigma(\varphi) - \varphi}{r}, \]
and hence by the translation invariance of $d(\cdot, \cdot)$
\begin{equation}\label{eq: F dist equivariance} d(F\circ\sigma(\psi_1), F\circ\sigma(\psi_2)) = d(\sigma\circ F(\psi_1), \sigma\circ F(\psi_2)).
\end{equation}
Furthermore, $F$ is a similarity scaling by $1/r$ in the metric $d$, i.e.
\begin{equation}\label{eq: F similarity}
d(F(\psi_1), F(\psi_2)) = \frac{1}{r}d(\psi_1, \psi_2).
\end{equation}
Take now $\rho \in (0, r)$ and set $\delta := \rho / r$. We claim that
\begin{equation}\label{eq: dZ bound} N_{d_M}(B_{d_{\Z}}(\varphi,r), 2\rho) \leq N_{d_M}(F(B_{d_{\Z}}(\varphi,r)), \delta).
\end{equation}
To see that, let $U_1, \ldots, U_N$ be an open cover of $F(B_{d_{\Z}}(\varphi,r))$ with sets of diameter at most $\delta$ in the metric $d_M$ and such that $N = N_{d_M}(F(B_{d_{\Z}}(\varphi,r)), \delta)$. Pick $F(\psi_i) \in U_i$ with $\psi_i \in B_{d_{\Z}}(\varphi,r)$ for each $i = 1, \ldots, N$. We claim that collection $\{ B_{d_M}(\psi_i, \rho) : i = 1, \ldots, N \}$ is a cover of $B_{d_{\Z}}(\varphi,r)$. Indeed, given $\psi \in B_{d_{\Z}}(\varphi,r)$, pick $i \in \{1, \ldots, N\}$ such that $d_M(F(\psi), F(\psi_i)) < \delta$. Then by \eqref{eq: F dist equivariance} and \eqref{eq: F similarity}
\[
\begin{split}
d_M(\psi, \psi_i) &= \max \limits_{0 \leq n < M} \{ d(\sigma^n \psi, \sigma^n \psi_i) \} = r \max \limits_{0 \leq n < M} \{ d(F(\sigma^n \psi), F(\sigma^n \psi_i)) \}\\
& = r \max \limits_{0 \leq n < M} \{ d(\sigma^n ( F( \psi)), \sigma^n( F( \psi_i))) \} \\
& = rd_M(F(\psi), F(\psi_i)) \leq r\delta = \rho.
\end{split}\]
This proves \eqref{eq: dZ bound}. If $\delta = \rho / r \leq \delta_0$, then \eqref{eq: dZ bound} together with \eqref{eq: mmdim upper bound} gives for $M \geq M(r, \rho) := M(\delta)$
\[ N_{d_M}(B_{d_{\Z}}(\varphi,r), 2\rho) \leq (r / \rho)^{M(\mmdim(B_1(V[a,b]), \sigma, d) + \eps)}. \]
On the other hand, if $\delta \geq \delta_0$, then \eqref{eq: dZ bound} and \eqref{eq: mmdim upper bound} give (as $0 < \rho < r$)
\[
\begin{split}
N_{d_M}(B_{d_{\Z}}(\varphi,r), 2\rho) & \leq N_{d_M}(F(B_{d_{\Z}}(\varphi,r)), \delta) \leq N_{d_M}(F(B_{d_{\Z}}(\varphi,r)), \delta_0) \\
& \leq (1 / \delta_0)^{ M(\mmdim(B_1(V[a,b]), \sigma, d) + \eps)} \\
& \leq (r / \rho) ^{M(\mmdim(B_1(V[a,b]), \sigma, d) + \eps)} (1 / \delta_0 ) ^{M(\mmdim(B_1(V[a,b]), \sigma, d) + \eps)}.
\end{split}
\]
The last two inequalities and Proposition \ref{prop: madim stable def}  imply  \eqref{eq: bd lim madim leq mmdim} and hence the proof is finished.
\end{proof}

\section{Mean Assouad dimension and spectrum of Bedford--McMullen carpet systems}\label{sec:carpet systems}
\subsection{Definitions and results}

Beyond the basic properties and examples, we determine the precise formulae for the mean Assouad dimension and spectrum of any Bedford--McMullen carpet systems, complementing the results in \cite{Tsu25carpets} where the mean Hausdorff dimension and metric mean dimension of such systems are investigated.
First, we recall the notation set up by Tsukamoto \cite{Tsu25carpets}. Let $a> b\geq 2$ be two natural numbers and set 
\[  A = \{0,1,\dots, a-1\}, \quad B = \{0,1,\dots, b-1\}. \]
Let $(A\times B)^\mathbb{N}$ be the one-sided full-shift defined on the alphabet $A\times B$
equipped with the shift transformation $\sigma: (A\times B)^{\mathbb{N}}\to (A\times B)^{\mathbb{N}}$.
Let $\pi: (A\times B)^{\mathbb{N}}\to B^{\mathbb{N}}$ be the natural projection.
By abuse of notation, we also denote by $\sigma:B^{\mathbb{N}}\to B^{\mathbb{N}}$ the shift map on $B^{\mathbb{N}}$.

Let $[0,1]^\mathbb{N}$ be the Hilbert cube. Consider the product
$[0,1]^{\mathbb{N}}\times [0,1]^{\mathbb{N}}$ equipped with the standard product metric $d$ by 
\begin{equation}  \label{eq: metric on infinite dimensional carpets}
   d\left((x, y), (x^\prime, y^\prime)\right)  = \sum_{n=1}^\infty 2^{-n} \max\left\{|x_n-x^\prime_n|, |y_n-y^\prime_n|\right\}, 
\end{equation}   
where $x = (x_n)_{n\in \mathbb{N}}, y= (y_n)_{n\in \mathbb{N}}$ and 
$x^\prime = (x^\prime_n)_{n\in \mathbb{N}}, y^\prime = (y^\prime_n)_{n\in \mathbb{N}}$ are points in $[0,1]^\mathbb{N}$.
We define a shift map on $[0,1]^{\mathbb{N}}\times [0,1]^{\mathbb{N}}$ 
%(also denoted by $\sigma: [0,1]^{\mathbb{N}}\times [0,1]^{\mathbb{N}} \to [0,1]^{\mathbb{N}}\times [0,1]^{\mathbb{N}}$) 
by 
\[  \sigma\left((x_n)_{n\in \mathbb{N}}, (y_n)_{n\in \mathbb{N}}\right) 
    = \left((x_{n+1})_{n\in \mathbb{N}}, (y_{n+1})_{n\in \mathbb{N}}\right). \]

Let $\Omega\subset (A\times B)^{\mathbb{N}}$ be a subshift, namely a closed subset with $\sigma(\Omega) \subset \Omega$.
We always assume that $\Omega$ is non-empty. We consider the restriction of $\pi:(A\times B)^{\mathbb{N}}\to B^{\mathbb{N}}$ to $\Omega$ and still denote it by $\pi:\Omega\to \pi(\Omega)$. Set $\Omega^\prime = \pi(\Omega)$, which is a subshift of $B^{\mathbb{N}}$. 
We define a \textbf{Bedford--McMullen carpet system} $X_\Omega \subset [0,1]^{\mathbb{N}}\times [0,1]^{\mathbb{N}}$ by\footnote{ Let us emphasize that a Bedford--McMullen carpet system $X_\Om$ is an infinite-dimensional subset of the infinite-dimensional space $[0,1]^\N \times [0,1]^\N$ and not a subset of $[0,1]^2$. The projection of $X_\Om$ onto the first $2N$ coordinates gives a subset of $[0,1]^N \times [0,1]^N$ which is a self-affine set of a Bedford--McMullen carpet type. The connection between finite dimensional projections and the system $X_\Om$ is explained with more details in Section \ref{sec: proj adim BM}.}
\[  X_\Omega = \left\{\left(\sum_{m=1}^\infty \frac{x_m}{a^m}, \sum_{m=1}^\infty \frac{y_m}{b^m}\right) 
      \in [0,1]^{\mathbb{N}} \times [0,1]^{\mathbb{N}} \middle|\, 
     (x_m, y_m)\in \Omega \text{ for all $m \geq 1$}\right\}. \]
Here for $x_m\in A^{\mathbb{N}}\subset \ell^\infty$, we consider the summation 
$\sum_{m=1}^\infty \frac{x_m}{a^m}$ in $\ell^\infty$. 
Then $\sum_{m=1}^\infty \frac{x_m}{a^m} \in [0,1]^{\mathbb{N}}$.
The same idea is applied to the term $\sum_{m=1}^\infty \frac{y_m}{b^m}$.
In this way we obtain a $\sigma$-invariant subset $X_{\Omega}$.
Then $(X_\Omega, \sigma)$ is a subsystem of 
$\left([0,1]^{\mathbb{N}}\times [0,1]^{\mathbb{N}}, \sigma\right)$.

\begin{thm} \label{theorem: mean Hausdorff dimension and metric mean dimension of carpet systems}\cite[Theorem 5.3]{Tsu25carpets}
The mean Hausdorff dimension and metric mean dimension of 
$(X_\Omega, \sigma, d)$ are given by 
\begin{equation}\label{eq: metric mean dimension of carpet systems}
    \begin{aligned}
   \mhdim\left(X_\Omega, \sigma, d\right) & =  \frac{\htop^{\log_a b}\left(\pi, \Omega, \sigma\right)}{\log b}, \\
   \mmdim\left(X_\Omega, \sigma, d\right) & = \frac{\htop\left(\Omega, \sigma\right)}{\log a} + 
      \left(\frac{1}{\log b} - \frac{1}{\log a}\right) \htop\left(\Omega^\prime, \sigma\right),
\end{aligned}
\end{equation}
where $\htop^{\log_a b}\left(\pi, \Omega, \sigma\right)$ is the weighted topological entropy\footnote{We refrain from defining weighted topological entropy since we do not discuss the mean Hausdorff dimension of $\left(X_\Omega, \sigma, d\right)$ in this paper. Instead, we refer the reader to \cite{FH16,Tsuweighted} for two different approaches to define the weighted topological entropy.} of the factor map 
$\pi:(\Omega, \sigma) \to \left(\Omega^\prime, \sigma\right)$ with respect to the weight $\log_a b$.
\end{thm}

We calculate the mean Assouad dimension and spectrum of a carpet system with respect to the metric \eqref{eq: metric on infinite dimensional carpets}.
\begin{thm} \label{theorem: mean Assouad dimension of carpet systems}
The mean Assouad dimension of $(X_\Omega, \sigma, d)$ is given by 
\begin{equation}\label{eq: mean Assouad dimension of carpet systems}
   \madim\left(X_\Omega, \sigma, d\right)  = \frac{\htop(\Omega\vert\Omega', \sigma)}{\log a} + 
      \frac{\htop(\Omega', \sigma)}{\log b}.
\end{equation}
Moreover, for all $\theta\in(0,1)$ we have
\begin{equation*}\label{eq: mean Assouad spectrum of carpet systems}
\begin{split}
    &\madim^{\theta}\left(X_\Omega, \sigma, d\right)  \\
    &= \frac{\mmdim\left(X_\Omega, \sigma, d\right)-\theta\left(\left(\frac{1}{\log a}-\frac{1}{\log b}\right)\htop(\Omega\vert\Omega', \sigma)+\frac{\htop(\Omega,\sigma)}{\log b}\right)}{1-\theta}
    \wedge \  \madim\left(X_\Omega, \sigma, d\right),
\end{split}
\end{equation*}
where $\htop(\Omega\vert\Omega', \sigma)$ is the topological conditional entropy\footnote{See Subsection \ref{defn:topological conditional entropy}.} of the factor map 
$\pi:(\Omega, \sigma) \to \left(\Omega^\prime, \sigma\right)$.
\end{thm}

It is well-known the Minkowski and Assouad dimensions of a planar Bedford--McMullen carpet coincide if and only if the carpet has \textit{uniform fibres}. We exhibit a dynamical counterpart for carpet systems. 
\begin{cor}
    For a carpet system $(X_{\Omega},\sigma)$ equipped with the metric \eqref{eq: metric on infinite dimensional carpets}, equality $\mmdim\left(X_\Omega, \sigma, d\right)=\madim\left(X_\Omega, \sigma, d\right)$ holds if and only if $\htop(\Omega,\sigma)=\htop(\Omega',\sigma)+\htop(\Omega\vert\Omega',\sigma)$.
    Moreover, if $\mmdim\left(X_\Omega, \sigma, d\right)=\madim\left(X_\Omega, \sigma, d\right)$, then $\mhdim\left(X_\Omega, \sigma, d\right)=\mmdim\left(X_\Omega, \sigma, d\right)$.
    If $\mhdim\left(X_\Omega, \sigma, d\right)=\mmdim\left(X_\Omega, \sigma, d\right)$ and the function $\Omega'\ni y\mapsto \htop(\pi^{-1}(y),\sigma)$ is constant, then $\mmdim\left(X_\Omega, \sigma, d\right)=\madim\left(X_\Omega, \sigma, d\right)$.
\end{cor}
\begin{proof}
    It is direct from \eqref{eq: metric mean dimension of carpet systems} and \eqref{eq: mean Assouad dimension of carpet systems} that $\mmdim\left(X_\Omega, \sigma, d\right)=\madim\left(X_\Omega, \sigma, d\right)$ if and only if 
    \begin{equation}\label{eq: mdimm = mdima}
    \htop(\Omega,\sigma)=\htop(\Omega',\sigma)+\htop(\Omega\vert\Omega',\sigma).
    \end{equation}

    For a \textit{TDS} $(X,T)$, denote by $\mathcal{M}^{T}(X)$ the set of $T$-invariant Borel probability measures on $X$ and denote by $h_{\mu}(X,T)$ the Kolmogorov--Sinal entropy of the measure $\mu\in\mathcal{M}^{T}(X)$.
    By the variational principle for weighted topological entropy (\cite[Theorem 1.3]{Tsuweighted}),
    \[\htop^{\log_a b}\left(\pi, \Omega, \sigma\right)=\sup_{\mu\in\mathcal{M}^{\sigma}(\Omega)}\{(\log_a b)h_{\mu}(\Omega,\sigma)+(1-\log_a b)h_{\pi_*\mu}(\Omega',\sigma)\}\]
    and therefore
    \[\mhdim\left(X_\Omega, \sigma, d\right)=\sup_{\mu\in\mathcal{M}^{\sigma}(\Omega)}\left\{\frac{h_{\mu}(\Omega,\sigma)}{\log a}+\left(\frac{1}{\log b}-\frac{1}{\log a}\right)h_{\pi_*\mu}(\Omega',\sigma)\right\},\]
    where $\pi_*\mu$ is the push-forward measure of $\mu$ by $\pi:\Omega\to\Omega'$. By the classical variational principle for topological entropy,
    \[\mmdim\left(X_\Omega, \sigma, d\right)  = \sup_{\mu\in\mathcal{M}^{\sigma}(\Omega)}\frac{h_{\mu}(\Omega,\sigma)}{\log a} + 
      \left(\frac{1}{\log b} - \frac{1}{\log a}\right) \sup_{\nu\in\mathcal{M}^{\sigma}(\Omega')}h_{\nu}(\Omega',\sigma).\]
      We conclude that $\mhdim\left(X_\Omega, \sigma, d\right)=\mmdim\left(X_\Omega, \sigma, d\right)$ if and only if there exists $\mu\in\mathcal{M}^{\sigma}(\Omega)$ satisfying
      \begin{equation}\label{eq: mdimm = mdimh}
          \htop(\Omega,\sigma)=h_{\mu}(\Omega,\sigma) \text{ and } \htop(\Omega',\sigma)=h_{\pi_*\mu}(\Omega',\sigma).
      \end{equation}
      We now assume that $\mmdim\left(X_\Omega, \sigma, d\right)=\madim\left(X_\Omega, \sigma, d\right)$. We choose $\mu\in\mathcal{M}^{\sigma}(\Omega)$ such that $\htop(\Omega,\sigma)=h_{\mu}(\Omega,\sigma)$. Let $\nu=\pi_*\mu$. By the relative variational principle (\cite[(1.2)]{LedWal77}),
      \begin{equation*}
          \begin{split}
              \htop(\Omega,\sigma) &= h_{\mu}(\Omega,\sigma) = \sup_{\eta\in\mathcal{M}^{\sigma}(\Omega): \pi_*\eta=\nu} h_{\eta}(\Omega,\sigma)\\
              &= h_{\nu}(\Omega',\sigma)+\int_{\Omega'}\htop(\pi^{-1}(y),\sigma)d\nu(y)\\
              &\leq \htop(\Omega',\sigma) + \sup_{y\in\Omega'}\htop(\pi^{-1}(y),\sigma)\\
              &\overset{\eqref{eq:fiber entropy}}{=}\htop(\Omega',\sigma) + \htop(\Omega\vert\Omega',\sigma)\overset{\eqref{eq: mdimm = mdima}}{=}\htop(\Omega,\sigma).
          \end{split}
      \end{equation*}
      We obtain that $h_{\nu}(\Omega',\sigma)+\int_{\Omega'}\htop(\pi^{-1}(y),\sigma)d\nu(y)
              = \htop(\Omega',\sigma) + \sup_{y\in\Omega'}\htop(\pi^{-1}(y),\sigma)$ and thus $\htop(\Omega',\sigma)=h_{\nu}(\Omega',\sigma)$. Therefore \eqref{eq: mdimm = mdimh} holds, hence $\mhdim\left(X_\Omega, \sigma, d\right)=\mmdim\left(X_\Omega, \sigma, d\right)$ as desired.

      We now assume that $\mhdim\left(X_\Omega, \sigma, d\right)=\mmdim\left(X_\Omega, \sigma, d\right)$ and $\htop(\pi^{-1}(y),\sigma)=\htop(\pi^{-1}(y'),\sigma)$ for all $y,y'\in\Omega'$. Then we may choose $\mu\in\mathcal{M}^{\sigma}(\Omega)$ such that \eqref{eq: mdimm = mdimh} holds. Let $\nu=\pi_*\mu$. Once again, by the relative variational principle
      \begin{equation*}
          \begin{split}
              \htop(\Omega,\sigma) &= h_{\mu}(\Omega,\sigma)= \sup_{\eta\in\mathcal{M}^{\sigma}(\Omega): \pi_*\eta=\nu} h_{\eta}(\Omega,\sigma)\\
              &= h_{\nu}(\Omega',\sigma)+\int_{\Omega'}\htop(\pi^{-1}(y),\sigma)d\nu(y)\\
              &= \htop(\Omega',\sigma) + \sup_{y\in\Omega'}\htop(\pi^{-1}(y),\sigma)\\
              &\overset{\eqref{eq:fiber entropy}}{=}\htop(\Omega',\sigma) + \htop(\Omega\vert\Omega',\sigma)
          \end{split}
      \end{equation*}
      and hence $\mmdim\left(X_\Omega, \sigma, d\right)=\madim\left(X_\Omega, \sigma, d\right)$, which finishes the proof.
\end{proof}

Observe that, by rearranging the formulae given in Theorems \ref{theorem: mean Hausdorff dimension and metric mean dimension of carpet systems} and \ref{theorem: mean Assouad dimension of carpet systems}, we immediately obtain an alternative expression of the mean Assouad spectrum of $(X_\Omega, \sigma, d)$ as follows.
\begin{cor}\label{cor: BM spectrum alternative}
    The mean Assouad spectrum of $(X_\Omega, \sigma, d)$ is completely determined by the ratio $\log b/\log a$, the metric mean dimension and the mean Assouad dimension of $(X_\Omega, \sigma, d)$. More precisely, we have
    \begin{equation*}
      \begin{split}
    &\madim^{\theta}\left(X_\Omega, \sigma, d\right)  \\
    &= \frac{\mmdim\left(X_\Omega, \sigma, d\right)-\theta\left(\madim\left(X_\Omega, \sigma, d\right)-\left(\madim\left(X_\Omega, \sigma, d\right)-\mmdim\left(X_\Omega, \sigma, d\right)\right)\frac{\log a}{\log b}\right)}{1-\theta}
      \end{split}
    \end{equation*}
    if $\theta\in(0,\log b/\log a]$, and
    \begin{equation*}
        \madim^{\theta}\left(X_\Omega, \sigma, d\right)=\madim\left(X_\Omega, \sigma, d\right)
    \end{equation*}
    if $\theta\in(\log b/\log a,1)$.
\end{cor}

\begin{rem}
    Recall Proposition \ref{bi-Lipschitz invariants} yields that bi-Lipschitz isomorphism preserves the mean Assouad spectrum. As a by-product of the formula in Corollary \ref{cor: BM spectrum alternative}, we see that the ratio $\log b/\log a$ is a bi-Lipschitz invariant of dynamical systems in the class of Bedford--McMullen carpet systems equipped with the metric defined by \eqref{eq: metric on infinite dimensional carpets}.
\end{rem}

\subsection{Preparations for the proofs}\label{preparations}
 Fix $N\in\N$ and denote by $\Omega|_N$ and $\Omega^\prime|_N$ the images of $\Omega$ and $\Omega^\prime$ 
under the projections
\begin{equation*}\label{projection I}
\begin{split}
    (A\times B)^{\mathbb{N}} & \to  A^N \times B^N,  \quad 
     \left((u_n)_{n\in \mathbb{N}}, (v_n)_{n\in \mathbb{N}}\right) \mapsto \left((u_1, \dots, u_N), (v_1, \dots, v_N)\right), \\
  B^{\mathbb{N}} & \to B^N,  \quad    (v_n)_{n\in \mathbb{N}} \mapsto (v_1, \dots, v_N).
\end{split}
\end{equation*}
Let $\pi_N:\Omega|_N\to\Omega^\prime|_N$ be the natural projection map. 

We start to calculate the topological conditional entropy $\htop(\Omega\vert\Omega', \sigma)$ of the factor map 
$\pi:(\Omega, \sigma) \to \left(\Omega^\prime, \sigma\right)$, which will be used later.
\begin{lem}\label{formula:topological conditional entropy of subshifts}
    In the above setting, we have
    \begin{equation*}
        \htop(\Omega\vert\Omega', \sigma)=\lim_{N\to\infty}\frac{\sup_{v\in \Omega'|_{N}}\log|\pi_{N}^{-1}(v)|}{N},
    \end{equation*}
    where $|\pi_N^{-1}(v)|$ is the cardinality of $\pi_N^{-1}(v)\subset\Omega\vert_N$.
    
    Notice that the above limit exists since $\sup_{v\in\Omega'\vert_N}|\pi_N^{-1}(v)|$ is sub-multiplicative in $N$.
\end{lem}
\begin{proof}
We define a metric on $(A\times B)^{\N}$ by
\begin{equation*}
    \dd((x,y),(x',y'))=2^{-\min\{n:(x_n,y_n)\neq (x'_n,y'_n)\}}.
\end{equation*}
Fix $y\in\Omega'\subset B^{\N}$. Take $m = m(\eps)\in\N$ with $2^{-m}<\varepsilon$. Let $v = v(y) =(y_1,\ldots,y_{N+m})\in\Omega'|_{N+m}$. For each $u\in A^{N+m}$ satisfying $(u,v)\in\Omega|_{N+m}\subset A^{N+m}\times B^{N+m}$, set 
\[
    U_u=\{(x,y)\in\Omega|(x_1,\ldots,x_{N+m})=u \}.
\]
Then 
\[\pi^{-1}(y)=\bigcup_{\substack{u\in A^{N+m} \\ \text{with } (u,v)\in\Omega|_{N+m}}}U_u\]
with $\diam(U_u,\dd_N)\leq\varepsilon$. Thus $N_{\dd_N}(\pi^{-1}(y),\varepsilon)\leq|\pi_{N+m}^{-1}(v)|$ for  $m=m(\eps)$ as above and
\begin{equation*}
    \begin{split}
        \htop(\Omega\vert\Omega', \sigma)&=\lim_{\varepsilon\to0}\left(\lim_{N\to\infty}\frac{\sup_{y\in \Omega'}\log N_{\dd_N}(\pi^{-1}(y),\varepsilon)}{N}\right)\\
        &\leq \lim_{\varepsilon\to0} \left( \lim_{N\to\infty}\frac{\sup_{y\in \Omega'}\log|\pi_{N+m(\eps)}^{-1}(v)|}{N} \right)  \\
        & =  \lim_{\varepsilon\to0} \left( \lim_{N\to\infty}\frac{\sup_{y\in \Omega'}\log|\pi_{N+m(\eps)}^{-1}(v)|}{N+m(\eps)} \right)  \\
        %&=\lim_{N\to\infty}\frac{\sup_{v\in \Omega'|_{N+m}}\log|\pi_{N+m}^{-1}(v)|}{N}
        & =\lim_{N\to\infty}\frac{\sup_{v\in \Omega'|_{N}}\log|\pi_{N}^{-1}(v)|}{N}.
    \end{split}
\end{equation*}

Next, let $0<\varepsilon<\frac{1}{2}$. Let $q = q(y)=(y_1,\ldots,y_N)$. Fix $\xi\in A^{\N}$ with $(\xi,y)\in\Omega$. We consider points in $\pi^{-1}(y)$  of the form $(x,y)$, where $((x_1,\ldots,x_N),q)\in\Omega|_N$ and $x_n=\xi_n$ for $n\geq N+1$. These points form an $\varepsilon$-separated set in $\pi^{-1}(y)$ with respect to the metric $\dd_N$. There are exactly $|\pi_N^{-1}(q)|$ of such choices. Then $\widehat{N}_{\dd_N}(\pi^{-1}(y),\varepsilon)\geq|\pi_{N}^{-1}(q)|$ and
\begin{equation*}
    \begin{split}
        \htop(\Omega\vert\Omega', \sigma)&=\lim_{\varepsilon\to0}\left(\limsup_{N\to\infty}\frac{\sup_{y\in \Omega'}\log \widehat{N}_{\dd_N}(\pi^{-1}(y),\varepsilon)}{N}\right)\\
        &\geq \lim_{N\to\infty}\frac{\sup_{y\in \Omega'}\log|\pi_{N}^{-1}(q)|}{N}\\
        &=\lim_{N\to\infty}\frac{\sup_{q\in \Omega'|_{N}}\log|\pi_{N}^{-1}(q)|}{N},
    \end{split}
\end{equation*}
which completes the proof.
\end{proof}

We also define $X_\Omega |_N$ as the image of $X_\Omega$ under the projection (denoted by $\tau_N$)
\begin{equation*}\label{projection II}
    [0,1]^{\mathbb{N}}\times [0,1]^{\mathbb{N}} \to [0,1]^N \times [0,1]^N, \quad 
     \left((x_n)_{n\in \mathbb{N}}, (y_n)_{n\in \mathbb{N}}\right) \mapsto \left((x_1, \dots, x_N), (y_1, \dots, y_N)\right).
\end{equation*}
We have 
\begin{equation}\label{eq: X Omega N}
X_\Omega|_N = \left\{\left(\sum_{m=1}^\infty \frac{x_m}{a^m}, \sum_{m=1}^\infty \frac{y_m}{b^m}\right)
     \in [0,1]^N\times [0,1]^N\middle|\, (x_m,y_m)\in \Omega|_N \text{ for all } m\geq 1\right\}.
\end{equation}
For $0<\rho<r$, we define
\begin{equation*}
    S_{\infty}(X_{\Omega},r,\rho)=\lim_{N\to\infty}\frac{1}{N}\sup_{(\widetilde{x},\widetilde{y})\in X_{\Omega}} \log N_{d_{\infty}^{N}}(B_{d_{\infty}^{N}}(\tau_N(\widetilde{x},\widetilde{y}),r),\rho),
\end{equation*}
where 
\begin{equation*}
    B_{d_{\infty}^{N}}(\tau_N(\widetilde{x},\widetilde{y}),r)=\{(\xi,\eta)\in X_{\Omega}\vert_N: d_{\infty}^{N}(\tau_N(\widetilde{x},\widetilde{y}),(\xi,\eta))\leq r\}
\end{equation*}
is the ball in $X_{\Omega}\vert_N$ centred at $\tau_N(\widetilde{x},\widetilde{y})$ and of radius $r$ and $d_{\infty}^{N}$ is the $\ell^\infty$-distance on $X_\Omega|_N \subset \mathbb{R}^{2N}$. It is easy to check that the quantity $\sup_{(\widetilde{x},\widetilde{y})\in X_{\Omega}} \log N_{d_{\infty}^{N}}(B_{d_{\infty}^{N}}(\tau_N(\widetilde{x},\widetilde{y}),r),\rho)$ is sub-additive in $N$ for arbitrary $0<\rho<r$.

\begin{lem}\label{lem:equivalent definition using projection and infinity norm}
    If we define
\begin{equation*}
    \madim(X_\Omega)=\inf\{s>0: \underset{C>0}{\exists}\ \underset{0<\rho<r}{\forall}\  e^{S_{\infty}(X_\Omega,r,\rho)}\leq C(r/\rho)^s\}
\end{equation*}
and
\begin{equation*}
    \madim^{\theta}(X_\Omega)=\inf\{s>0: \underset{C>0}{\exists}\ \underset{0<r<1}{\forall}\  e^{S_{\infty}(X_\Omega,r,r^{1/\theta})}\leq C(r/r^{1/\theta})^s\}
\end{equation*} 
for $\theta\in(0,1)$.
    Then it holds $\madim(X_\Omega, \sigma, d)=\madim(X_\Omega)$ and $\madim^{\theta}(X_\Omega, \sigma, d)=\madim^{\theta}(X_\Omega)$ for all $\theta\in(0,1)$. 
\end{lem}

We omit the proof of Lemma \ref{lem:equivalent definition using projection and infinity norm} since it is essentially the same to that of Lemma \ref{equivalent defn}.

Let $(x, y) \in \left(\Omega|_N\right)^{\mathbb{N}}$ where $x = (x_m)_{m\in \mathbb{N}}$ and $y = (y_m)_{m\in \mathbb{N}}$ with 
$x_m\in A^N$, $y_m\in B^N$ and $(x_m, y_m) \in \Omega|_N$. For $r>0$, define $l_1(r),\ l_2(r)$ be the unique integers satisfying
\begin{equation*}
\begin{split}
    a^{-l_1(r)} & \leq r < a^{-l_1(r)+1},\\
    b^{-l_2(r)} & \leq r < b^{-l_2(r)+1}.
\end{split}
\end{equation*}
In particular,
\begin{equation}\label{eq: l1 l2 def}
    -\frac{\log r}{\log a} \leq l_1(r) < -\frac{\log r}{\log a}+1, \ -\frac{\log r}{\log b} \leq l_2(r) < -\frac{\log r}{\log b}+1.
\end{equation}
We denote the $N$-th \textbf{approximate square} centred at $(x,y)$ with side-length $r>0$ by $Q_{N,r}(x,y)$ and define it by
\begin{equation}  \label{eq: definition of Q_Nr}
    Q_{N, r}(x, y)  =
   \left\{\left(\sum_{m=1}^\infty \frac{x^\prime_m}{a^m}, \sum_{m=1}^\infty \frac{y^\prime_m}{b^m}\right)  \middle|\, 
    \parbox{2.5in}{\centering $(x^\prime_m, y^\prime_m) \in \Omega|_N$ for all $m\geq 1$ with \\
    $x^\prime_m = x_m\ \text{ for } 1\leq m \leq l_1(r))$ and \\ $y^\prime_m = y_m\ \text{ for } 1\leq m \leq l_2(r))$} \right\} \subset [0,1]^N \times [0,1]^N. 
\end{equation}   
Then a direct calculation shows that
\begin{equation*}   \label{eq: diameter of Q_NM}
    \diam\left(Q_{N,r}(x,y), d_{\infty}^{N}\right) \leq \max\left\{a^{-l_1(r)}, b^{-l_2(r)}\right\}\leq r.
\end{equation*}

We now define
\begin{equation*}
    \widetilde{S}(X_\Omega,r,\rho)=\lim_{N\to\infty}\frac{1}{N}\sup_{(x, y) \in \left(\Omega|_N\right)^{\mathbb{N}}} \log {\tilde{N}}_{d_{\infty}^{N}}(Q_{N,r}(x,y),\rho),
\end{equation*}
where ${\tilde{N}}_{d_{\infty}^{N}}(Q_{N,r}(x,y),\rho)$ denote the least number of elements in $\{Q_{N,\rho}(\xi,\eta): (\xi,\eta)\in \left(\Omega|_N\right)^{\N}\}$ required to cover $Q_{N,r}(x,y)$. We may replace $S_{\infty}(X_{\Omega},r,\rho)$ with $\widetilde{S}(X_\Omega,r,\rho)$ in the definition of $\madim(X_\Omega)$ and $\madim^{\theta}(X_\Omega)$.  This is because ball $B_{d_{\infty}^{N}}(\tau_N(x,y),r)$ and approximate square $Q_{N,r}(x,y)$ are comparable, similarly as in the case of non-dynamical Bedford-McMullen carpets \cite[Lemma 7.1]{Fra14}.
More precisely, it follows from Lemma \ref{lem:equivalent definition using projection and infinity norm}
\begin{equation*}
    \madim(X_\Omega)=\inf\{s>0: \underset{C>0}{\exists}\ \underset{0<\rho<r}{\forall}\  e^{\widetilde{S}(X_\Omega,r,\rho)}\leq C(r/\rho)^s\}.
\end{equation*} 
Moreover, for all $\theta\in(0,1)$,
\begin{equation*}
    \madim^{\theta}(X_\Omega)=\inf\{s>0: \underset{C>0}{\exists}\ \underset{0<r<1}{\forall}\  e^{\widetilde{S}(X_\Omega,r,r^{1/\theta})}\leq C(r/r^{1/\theta})^s\}.
\end{equation*}

\subsection{Calculation of the mean Assouad dimension: upper bound}
 Fix $N\in\N$ and let $(x, y) \in \left(\Omega|_N\right)^{\mathbb{N}}$ where $x = (x_m)_{m\in \mathbb{N}}$ and $y = (y_m)_{m\in \mathbb{N}}$ with $x_m\in A^N$, $y_m\in B^N$ and $(x_m, y_m) \in \Omega|_N$. Recall $\pi_N:\Omega|_N\to \Omega'|_N$ is the natural projection map.
 Choose $0<\rho< r$ small enough.
\begin{claim}\label{upper bound of approximate squares}
    There exists $C>0$ independent of the choice of $(x, y) \in \left(\Omega|_N\right)^{\mathbb{N}}$ such that for all $0<\rho<r$ we have
    \begin{equation*}
        {\tilde{N}}_{d_{\infty}^{N}}(Q_{N,r}(x,y),\rho) \leq C\left(\sup_{v\in\Omega'\vert_N}|\pi_N^{-1}(v)|\right)^{l_1(\rho)-l_1(r)} \left(|\Omega'\vert_N|^{l_2(\rho)-l_2(r)}\right).
    \end{equation*}
\end{claim}
We assume that the claim holds. 
Using inequalities \eqref{eq: l1 l2 def} and Lemma \ref{formula:topological conditional entropy of subshifts}, we obtain that
\begin{equation*}
    \begin{split}
        &\widetilde{S}(X_\Omega,r,\rho)=\lim_{N\to\infty}\frac{1}{N}\sup_{(x, y) \in \left(\Omega|_N\right)^{\mathbb{N}}} \log {\tilde{N}}_{d_{\infty}^{N}}(Q_{N,r}(x,y),\rho)\\
        &\leq (l_1(\rho)-l_1(r))\lim\limits_{N\to\infty}\frac{\sup_{v\in\Omega'\vert_N}\log|\pi_N^{-1}(v)|}{N} +(l_2(\rho)-l_2(r))\lim\limits_{N\to\infty}\frac{\log |\Omega'\vert_N|}{N}\\
        &=(l_1(\rho)-l_1(r)) \htop(\Omega\vert\Omega', \sigma) +(l_2(\rho)-l_2(r)) \htop(\Omega',\sigma)\\
        &\leq \left(\frac{\log r}{\log a}-\frac{\log \rho}{\log a}+1\right) \htop(\Omega\vert\Omega', \sigma)  +\left(\frac{\log r}{\log b}-\frac{\log \rho}{\log b}+1\right) \htop(\Omega',\sigma)\\
        &=\log\frac{r}{\rho}\left(\frac{\htop(\Omega\vert\Omega', \sigma)}{\log a}+\frac{\htop(\Omega',\sigma)}{\log b}\right)
        +\htop(\Omega\vert\Omega', \sigma)+\htop(\Omega',\sigma)
    \end{split}
\end{equation*}
and therefore
\begin{equation*}
    \madim\left(X_\Omega, \sigma, d\right)  \leq \frac{\htop(\Omega\vert\Omega', \sigma)}{\log a} + 
      \frac{\htop(\Omega', \sigma)}{\log b}.
\end{equation*}
Thus it suffices to prove Claim \ref{upper bound of approximate squares}.

\begin{proof}[Proof of Claim \ref{upper bound of approximate squares}]
We consider the following two cases.

    \textit{Case 1:} $l_1(r)\leq l_2(r) \leq l_1(\rho) \leq l_2(\rho)$.

    For each $m=l_1(r)+1,\ldots,l_2(r)$ choose $\xi_m\in A^N$ satisfying $(\xi_m,y_m)\in\Omega\vert_N$. We observe that $Q_{N,r}(x,y)$ is composed of several middle-sized (multidimensional) vertical rectangles of the form
    \begin{equation}\label{eq: approx square xi}  
   \left\{\left(\sum_{m=1}^\infty \frac{x^\prime_m}{a^m}, \sum_{m=1}^\infty \frac{y^\prime_m}{b^m}\right)\in Q_{N,r}(x,y)  \middle|\, 
    \parbox{2.5in}{\centering $x^\prime_m = \xi_m\ \text{ for }l_1(r)+1\leq m \leq l_2(r)$} \right\}, 
\end{equation}   
which are denoted by 
\begin{equation*}
   Q_{N,r}(x_1,\ldots,x_{l_1(r)},\xi_{l_1(r)+1},\ldots,\xi_{l_2(r)} ; y_1,\ldots,y_{l_2(r)}). 
\end{equation*}
These vertical rectangles share the same height with $Q_{N,r}(x,y)$, that is $b^{-l_2(r)}$, but have smaller width $a^{-l_2(r)}$.
 The total number of such middle-sized rectangles is exactly
\begin{equation*}
    \prod_{m=l_1(r)+1}^{l_2(r)} |\pi_N^{-1}(y_m)|.
\end{equation*}
Furthermore, inside each vertical rectangle $Q_{N,r}(x_1,\ldots,x_{l_1(r)},\xi_{l_1(r)+1},\ldots,\xi_{l_2(r)} ; y_1,\ldots,y_{l_2(r)})$, we choose $|\Om|_N|$ smaller rectangles. We iterate the construction for $l_1(\rho)-l_2(r)$ times  and finally obtain smaller-sized rectangles of the form
\begin{equation*}
    \left\{\left(\sum_{m=1}^\infty \frac{x^\prime_m}{a^m}, \sum_{m=1}^\infty \frac{y^\prime_m}{b^m}\right)\in Q_{N,r}(x_1,\ldots,x_{l_1(r)},\xi_{l_1(r)+1},\ldots,\xi_{l_2(r)} ; y_1,\ldots,y_{l_2(r)})   \middle|\, 
    \parbox{1.8in}{\centering $(x^\prime_m, y^\prime_m) = (\xi_m, \eta_m)\ \text{ for }$ \\  $l_2(r)+1\leq m \leq l_1(\rho)$} \right\} 
\end{equation*}
with $(\xi_m, \eta_m)\in\Omega\vert_N$ for each $l_2(r)+1\leq m \leq l_1(\rho)$, which are denoted by
\begin{equation*}
    Q_{N,r}(x_1,\ldots,x_{l_1(r)},\xi_{l_1(r)+1},\ldots,\xi_{l_1(\rho)} ; y_1,\ldots,y_{l_2(r)},\eta_{l_2(r)+1},\ldots,\eta_{l_1(\rho)}).
\end{equation*}
Up to now, we have a collection of rectangles with width $a^{-l_1(\rho)}$, approximately $\rho$, and height $b^{-l_1(\rho)}$, which is larger than $\rho$ since $l_1(\rho)\leq l_2(\rho)$.
The total number of such smaller-sized rectangles is exactly
\begin{equation*}
    |\Omega\vert_N|^{l_1(\rho)-l_2(r)}.
\end{equation*}
Finally, we may cover $Q_{N,r}(x_1,\ldots,x_{l_1(r)},\xi_{l_1(r)+1},\ldots,\xi_{l_1(\rho)} ; y_1,\ldots,y_{l_2(r)},\eta_{l_2(r)+1},\ldots,\eta_{l_1(\rho)})$ by $N$-th approximate squares of side-length $\rho$ of the form 
\begin{equation*}
    Q_{N, \rho}(\bar x, \bar y)  =
   \left\{\left(\sum_{m=1}^\infty \frac{x^\prime_m}{a^m}, \sum_{m=1}^\infty \frac{y^\prime_m}{b^m}\right)  \middle|\, 
    \parbox{2.5in}{\centering $(x^\prime_m, y^\prime_m) \in \Omega|_N$ for all $m\geq 1$ with \\
    $x^\prime_m = \bar x_m\ \text{ for } 1\leq m \leq l_1(\rho)$ and \\ $y^\prime_m = \bar y_m\ \text{ for }1\leq m \leq l_2(\rho)$} \right\}, 
\end{equation*}
where $(\bar x, \bar y)$ satisfies
\[
\begin{aligned}
\bar x_m=x_m & \text{ for } m=1,\ldots,l_1(r), \\
\bar x_m=\xi_m & \text{ for } m=l_1(r)+1,\ldots,l_1(\rho),\\
\bar y_m=y_m & \text{ for } 1\leq m\leq l_2(r),\\
\bar y_m=\eta_m & \text{ for }  l_2(r)+1\leq m\leq l_1(\rho), \\
\bar y_m\in\Omega'\vert_N & \text{ for }  l_1(\rho)+1\leq m\leq l_2(\rho).
\end{aligned}
\]
There are $|\Om'|_N|^{l_2(\rho) - l_1(\rho)}$ of such choices  as  we may cover smaller-sized rectangles simultaneously if they are located in the same row, so that we need $|\Omega'\vert_N|$ covering sets at each step.

Notice that the cover constructed above is optimal up to a constant. Putting these estimates together, we have
\begin{equation*}
    \begin{split}
        {\tilde{N}}_{d_{\infty}^{N}}(Q_{N,r}(x,y),\rho)&\asymp
        \left(\prod_{m=l_1(r)+1}^{l_2(r)} |\pi_N^{-1}(y_m)|\right) \left(|\Omega\vert_N|^{l_1(\rho)-l_2(r)}\right) \left(|\Omega'\vert_N|^{l_2(\rho)-l_1(\rho)}\right) \\
        &\leq \left(\sup_{v\in\Omega'\vert_N}|\pi_N^{-1}(v)|\right)^{l_2(r)-l_1(r)} \left(|\Omega\vert_N|^{l_1(\rho)-l_2(r)}\right) \left(|\Omega'\vert_N|^{l_2(\rho)-l_1(\rho)}\right) \\
        &\leq \left(\sup_{v\in\Omega'\vert_N}|\pi_N^{-1}(v)|\right)^{l_2(r)-l_1(r)} \left(|\Omega'\vert_N|\cdot\sup_{v\in\Omega'\vert_N}|\pi_N^{-1}(v)|\right)^{l_1(\rho)-l_2(r)} \left(|\Omega'\vert_N|^{l_2(\rho)-l_1(\rho)}\right) \\ 
        &=\left(\sup_{v\in\Omega'\vert_N}|\pi_N^{-1}(v)|\right)^{l_1(\rho)-l_1(r)} \left(|\Omega'\vert_N|^{l_2(\rho)-l_2(r)}\right).
    \end{split}
\end{equation*}

    \textit{Case 2:} $l_1(r) \leq l_1(\rho) \leq l_2(r) \leq l_2(\rho)$.

      For each $m=l_1(r)+1,\ldots,l_1(\rho)$ choose $\xi_m\in A^N$ satisfying $(\xi_m,y_m)\in\Omega\vert_N$. We observe that $Q_{N,r}(x,y)$ is composed of several rectangles of the form
    \begin{equation*}  
   \left\{\left(\sum_{m=1}^\infty \frac{x^\prime_m}{a^m}, \sum_{m=1}^\infty \frac{y^\prime_m}{b^m}\right)\in Q_{N,r}(x,y)  \middle|\, 
    \parbox{2.5in}{\centering $x^\prime_m = \xi_m\ \text{ for }l_1(r)+1\leq m \leq l_1(\rho)$} \right\}, 
\end{equation*}   
which are denoted by 
\begin{equation*}
   Q_{N,r}(x_1,\ldots,x_{l_1(r)},\xi_{l_1(r)+1},\ldots,\xi_{l_1(\rho)} ; y_1,\ldots,y_{l_2(r)}). 
\end{equation*}
 The total number of such rectangles is exactly
\begin{equation*}
    \prod_{m=l_1(r)+1}^{l_1(\rho)} |\pi_N^{-1}(y_m)|.
\end{equation*}
Moreover, we may cover $Q_{N,r}(x_1,\ldots,x_{l_1(r)},\xi_{l_1(r)+1},\ldots,\xi_{l_1(\rho)} ; y_1,\ldots,y_{l_2(r)})$ by approximate squares $Q_{N, \rho}(\bar x, \bar y)$ of side-length $\rho$,
where $(\bar x, \bar y)$ satisfies
\[
\begin{split}
\bar x_m=x_m & \text{ for } m=1,\ldots,l_1(r),\\
\bar x_m=\xi_m & \text{ for } m=l_1(r)+1,\ldots,l_1(\rho),\\
\bar y_m=y_m & \text{ for } 1\leq m\leq l_2(r),\\
\bar y_m\in\Omega'\vert_N & \text{ for } l_2(r)+1\leq m\leq l_2(\rho).
\end{split}
\]
There are $|\Omega'\vert_N|^{l_2(\rho)-l_2(r)}$ of those. Putting these estimates together, we have
\begin{equation}\label{easier case}
    \begin{split}
        {\tilde{N}}_{d_{\infty}^{N}}(Q_{N,r}(x,y),\rho)&\asymp
        \left(\prod_{m=l_1(r)+1}^{l_1(\rho)} |\pi_N^{-1}(y_m)|\right) \left(|\Omega'\vert_N|^{l_2(\rho)-l_2(r)}\right) \\
        &\leq \left(\sup_{v\in\Omega'\vert_N}|\pi_N^{-1}(v)|\right)^{l_1(\rho)-l_1(r)} \left(|\Omega'\vert_N|^{l_2(\rho)-l_2(r)}\right).
    \end{split}
\end{equation}

This completes the proof of Claim \ref{upper bound of approximate squares}.

\end{proof}

\subsection{Calculation of the mean Assouad dimension: lower bound}\label{Sec:mean Assouad dimension: lower bound}
Fix $N\in\N$.  Choose $0<\rho<r$ arbitrarily small and satisfying $(br)^{\frac{\log a}{\log b}}<\rho$. Then $l_1(r) \leq l_1(\rho) \leq l_2(r) \leq l_2(\rho)$. Moreover, let $(x, y) \in \left(\Omega|_N\right)^{\mathbb{N}}$ where $x = (x_m)_{m\in \mathbb{N}}$ and $y = (y_m)_{m\in \mathbb{N}}$ with $(x_m, y_m) \in \Omega|_N$ and $|\pi_N^{-1}(y_m)|=\sup_{v\in\Omega'\vert_N}|\pi_N^{-1}(v)|$ for $m=l_1(r)+1,\ldots,l_1(\rho)$. Then the unique inequality in \eqref{easier case} is replaced by equality and therefore
\begin{equation*}
    \begin{split}
        &\lim_{N\to\infty}\frac{1}{N} \log {\tilde{N}}_{d_{\infty}^{N}}(Q_{N,r}(x,y),\rho)\\
        &\asymp (l_1(\rho)-l_1(r))\lim\limits_{N\to\infty}\frac{\sup_{v\in\Omega'\vert_N}\log|\pi_N^{-1}(v)|}{N} +(l_2(\rho)-l_2(r))\lim\limits_{N\to\infty}\frac{\log |\Omega'\vert_N|}{N}\\
        &\geq \left(-\frac{\log \rho}{\log a}+\frac{\log r}{\log a}-1\right)\htop(\Omega\vert\Omega', \sigma)+
        \left(-\frac{\log \rho}{\log b}+\frac{\log r}{\log b}-1\right)\htop(\Omega',\sigma)\\
        % &=(l_1(\rho)-l_1(r)) \htop(\Omega\vert\Omega', \sigma) +(l_2(\rho)-l_2(r)) \htop(\Omega',\sigma)\\
        % &\asymp \left(\frac{\log r}{\log a}-\frac{\log \rho}{\log a}\right) \htop(\Omega\vert\Omega', \sigma)  +\left(\frac{\log r}{\log b}-\frac{\log \rho}{\log b}\right) \htop(\Omega',\sigma)\\
        &=\log\frac{r}{\rho}\left(\frac{\htop(\Omega\vert\Omega', \sigma)}{\log a}+\frac{\htop(\Omega',\sigma)}{\log b}\right)
        -\htop(\Omega\vert\Omega', \sigma)-\htop(\Omega',\sigma).
    \end{split}
\end{equation*}
Since we can choose a sequence of points $(x, y)\in (\Omega|_N)^{\N}$ and scales $0<\rho<r$ satisfying the above inequality, it follows that 
\begin{equation*}
    \madim\left(X_\Omega, \sigma, d\right)  \geq \frac{\htop(\Omega\vert\Omega', \sigma)}{\log a} + 
      \frac{\htop(\Omega', \sigma)}{\log b}.
\end{equation*}

\subsection{Calculation of the mean Assouad spectrum}
The following lemma is a prior indication that there will be a phase transition in the spectrum occurring at $\theta=\log b/\log a$.
\begin{lem}
    Let $r\in(0,1)$ and $\theta\in(0,1)$. 
    \begin{enumerate}[$(1)$]
        \item If $\theta\leq\log b/\log a$, then $l_1(r)\leq l_2(r)\leq l_1(r^{1/\theta})\leq l_2(r^{1/\theta})$.
        \item If $\theta\geq\log b/\log a$, then $l_1(r)\leq l_1(r^{1/\theta})\leq l_2(r)\leq l_2(r^{1/\theta})$.
    \end{enumerate}
\end{lem}
\begin{proof}
    It follows directly from the definition of $l_1(\cdot)$ and $l_2(\cdot)$.
\end{proof}

\begin{prop}\label{mean Assouad spectrum:part i}
    Let $\theta\in(0,\log b/\log a]$. Then
    \begin{equation*}
    \begin{split}
        &\madim^{\theta}\left(X_\Omega, \sigma, d\right)  \\
    &= \frac{\frac{\htop\left(\Omega, \sigma\right)}{\log a} + 
      \left(\frac{1}{\log b} - \frac{1}{\log a}\right) \htop\left(\Omega^\prime, \sigma\right)-\theta\left(\left(\frac{1}{\log a}-\frac{1}{\log b}\right)\htop(\Omega\vert\Omega', \sigma)+\frac{\htop(\Omega,\sigma)}{\log b}\right)}{1-\theta}.
    \end{split}
    \end{equation*}
\end{prop}
\begin{proof}
    Fix $N\in\N$ and let $(x, y) \in \left(\Omega|_N\right)^{\mathbb{N}}$ where $x = (x_m)_{m\in \mathbb{N}}$ and $y = (y_m)_{m\in \mathbb{N}}$ with $(x_m, y_m) \in \Omega|_N$. Fix $r\in(0,1)$. Let $Q_{N, r}(x, y)$ be as defined in \eqref{eq: definition of Q_Nr}. Then $Q_{N, r}(x, y)$ is composed of middle-sized rectangles $Q_{N,r}(x_1,\ldots,x_{l_1(r)},\xi_{l_1(r)+1},\ldots,\xi_{l_2(r)} ; y_1,\ldots,y_{l_2(r)})$ as in \eqref{eq: approx square xi}, with $\xi_m\in\pi_N^{-1}(y_m)$ for $m=l_1(r)+1,\ldots,l_2(r)$. The total number of such middle-sized rectangles is exactly
    \begin{equation*}
        \prod_{m=l_1(r)+1}^{l_2(r)}|\pi_N^{-1}(y_m)|.
    \end{equation*}
    Furthermore, each $Q_{N,r}(x_1,\ldots,x_{l_1(r)},\xi_{l_1(r)+1},\ldots,\xi_{l_2(r)} ; y_1,\ldots,y_{l_2(r)})$ is composed of smaller-sized rectangles
\begin{equation*}
    Q_{N,r}(x_1,\ldots,x_{l_1(r)},\xi_{l_1(r)+1},\ldots,\xi_{l_1(r^{1/\theta})} ; y_1,\ldots,y_{l_2(r)},\eta_{l_2(r)+1},\ldots,\eta_{l_1(r^{1/\theta})}).
\end{equation*}
 with $(\xi_m, \eta_m)\in\Omega\vert_N$ for each $l_2(r)+1\leq m \leq l_1(r^{1/\theta})$. The total number of such smaller-sized rectangles is exactly
\begin{equation*}
    |\Omega\vert_N|^{l_1(r^{1/\theta})-l_2(r)}.
\end{equation*}
Finally, we may cover $Q_{N,r}(x_1,\ldots,x_{l_1(r)},\xi_{l_1(r)+1},\ldots,\xi_{l_1(r^{1/\theta})} ; y_1,\ldots,y_{l_2(r)},\eta_{l_2(r)+1},\ldots,\eta_{l_1(r^{1/\theta})})$ by $N$-th approximate squares $Q_{N, r^{1/\theta}}(\bar x, \bar y)$,
where $(\bar x, \bar y)$ satisfies
\[
\begin{split}
\bar x_m=x_m & \text{ for } m=1,\ldots,l_1(r),\\
\bar x_m=\xi_m & \text{ for } m=l_1(r)+1,\ldots,l_1(r^{1/\theta}),\\
\bar y_m=y_m & \text{ for } 1\leq m\leq l_2(r),\\
\bar y_m=\eta_m & \text{ for } l_2(r)+1\leq m\leq l_1(r^{1/\theta}),\\
\bar y_m\in\Omega'\vert_N & \text{ for } l_1(r^{1/\theta})+1\leq m\leq l_2(r^{1/\theta}).
\end{split}
\]
There are $|\Omega'_N|^{l_2(r^{1 / \theta}) - l_1(r^{1 / \theta})}$ of such approximate squares. Putting these estimates together, we have
\begin{equation*}
    \begin{split}
        &{\tilde{N}}_{d_{\infty}^{N}}(Q_{N,r}(x,y),r^{1/\theta})\\
        &\asymp
        \left(\prod_{m=l_1(r)+1}^{l_2(r)} |\pi_N^{-1}(y_m)|\right) \left(|\Omega\vert_N|^{l_1(r^{1/\theta})-l_2(r)}\right) \left(|\Omega'\vert_N|^{l_2(r^{1/\theta})-l_1(r^{1/\theta})}\right) \\
        &\overset{(\star)}{\leq} \left(\sup_{v\in\Omega'\vert_N}|\pi_N^{-1}(v)|\right)^{l_2(r)-l_1(r)} \left(|\Omega\vert_N|^{l_1(r^{1/\theta})-l_2(r)}\right) \left(|\Omega'\vert_N|^{l_2(r^{1/\theta})-l_1(r^{1/\theta})}\right) \\
        &= \left(\sup_{v\in\Omega'\vert_N}|\pi_N^{-1}(v)|\right)^{\lfloor\log r/\log a\rfloor-\lfloor\log r/\log b\rfloor} \left(|\Omega\vert_N|^{\lfloor\log r/\log b\rfloor-\lfloor\log r/(\theta\log a)\rfloor}\right)\\
        &\cdot\left(|\Omega'\vert_N|^{\lfloor\log r/(\theta\log a)\rfloor-\lfloor\log r/\theta\log b\rfloor}\right).
    \end{split}
\end{equation*}
Taking logarithms, dividing by $N$ and taking $N\to\infty$, we obtain that
\begin{equation*}
    \begin{split}
        &\widetilde{S}(X_\Omega,r,r^{1/\theta})=\lim_{N\to\infty}\frac{1}{N}\sup_{(x, y) \in \left(\Omega|_N\right)^{\mathbb{N}}} \log {\tilde{N}}_{d_{\infty}^{N}}(Q_{N,r}(x,y),r^{1/\theta})\\
        &\leq \left(\frac{\log r}{\log a}-\frac{\log r}{\log b}+1\right) \htop(\Omega\vert\Omega', \sigma) +\left(\frac{\log r}{\log b}-\frac{\log r}{\theta\log a}+1\right) \htop(\Omega,\sigma) \\
        &+\left(\frac{\log r}{\theta\log a}-\frac{\log r}{\theta\log b}+1\right) \htop(\Omega',\sigma)\\
        &=\left(\left(\frac{1}{\log a}-\frac{1}{\log b}\right)\htop(\Omega\vert\Omega', \sigma)+\frac{\htop(\Omega,\sigma)}{\log b}-\frac{\htop(\Omega,\sigma)}{\theta\log a}+\left(\frac{1}{\theta\log a}-\frac{1}{\theta\log b}\right)\htop(\Omega', \sigma)\right) \log r\\
        &+\htop(\Omega\vert\Omega', \sigma)+\htop(\Omega,\sigma)+\htop(\Omega',\sigma)
    \end{split}
\end{equation*}
and therefore
\begin{equation*}
    \begin{split}
        &\madim^{\theta}\left(X_\Omega, \sigma, d\right) \\
        &\leq
        \frac{\left(\frac{1}{\log a}-\frac{1}{\log b}\right)\htop(\Omega\vert\Omega', \sigma)+\frac{\htop(\Omega,\sigma)}{\log b}-\frac{\htop(\Omega,\sigma)}{\theta\log a}+\left(\frac{1}{\theta\log a}-\frac{1}{\theta\log b}\right)\htop(\Omega', \sigma)}{1-1/\theta}\\
        &=\frac{\frac{\htop\left(\Omega, \sigma\right)}{\log a} + 
      \left(\frac{1}{\log b} - \frac{1}{\log a}\right) \htop\left(\Omega^\prime, \sigma\right)-\theta\left(\left(\frac{1}{\log a}-\frac{1}{\log b}\right)\htop(\Omega\vert\Omega', \sigma)+\frac{\htop(\Omega,\sigma)}{\log b}\right)}{1-\theta}.
    \end{split}
\end{equation*}
This yields the desired upper bound. Observe that it we choose $(x, y) \in \left(\Omega|_N\right)^{\mathbb{N}}$ with $|\pi_N^{-1}(y_m)|=\sup_{v\in\Omega'\vert_N}|\pi_N^{-1}(v)|$ for $m=l_1(r)+1,\ldots,l_2(r)$, then the inequality $(\star)$ in the above argument is replaced by equality and our covering estimates were optimal up to multiplicative constants. Therefore, in this way we will obtain the required lower bound for the mean Assouad spectrum of $(X_{\Omega},\sigma,d)$.
\end{proof}

\begin{prop}
     Let $\theta\in(\log b/\log a,1)$. Then
    \begin{equation*}
    \begin{split}
        \madim^{\theta}\left(X_\Omega, \sigma, d\right) = \frac{\htop(\Omega\vert\Omega', \sigma)}{\log a} + 
      \frac{\htop(\Omega', \sigma)}{\log b}.
    \end{split}
    \end{equation*}
\end{prop}
\begin{proof}
   % {\r [Adam: I think that it would more clear to give more details below]}
    
    We proceed with the covering argument in Proposition \ref{mean Assouad spectrum:part i}, whereas this time we will obtain a collection of rectangles with height approximately $r^{1/\theta}$ after $l_1(r^{1/\theta})$ steps, which is before the bases become smaller than $r$, since in this situation $l_1(r^{1/\theta})\leq l_2(r)$. For this reason, the middle term concerning power of $|\Omega_N|$ is not required anymore. More precisely, for $(x, y) \in \left(\Omega|_N\right)^{\mathbb{N}}$, let $Q_{N, r}(x, y)$ be defined as in \eqref{eq: definition of Q_Nr}. Then $Q_{N, r}(x, y)$ is composed of middle-sized rectangles of the form
    \begin{equation*}
   \left\{\left(\sum_{m=1}^\infty \frac{x^\prime_m}{a^m}, \sum_{m=1}^\infty \frac{y^\prime_m}{b^m}\right)\in Q_{N,r}(x,y)  \middle|\, 
    \parbox{2.5in}{\centering $x^\prime_m = \xi_m\ \text{ for }l_1(r)+1\leq m \leq l_1(r^{1/\theta})$} \right\}, 
\end{equation*} 
    which are denoted by 
\begin{equation*}
   Q_{N,r}(x_1,\ldots,x_{l_1(r)},\xi_{l_1(r)+1},\ldots,\xi_{l_1(r^{1/\theta})} ; y_1,\ldots,y_{l_2(r)}). 
\end{equation*}
 The total number of such middle-sized rectangles (each with height approximately $r^{1/\theta}$) is exactly
\begin{equation*}
    \prod_{m=l_1(r)+1}^{l_1(r^{1/\theta})} |\pi_N^{-1}(y_m)|.
\end{equation*}
    Furthermore, we may cover each $Q_{N,r}(x_1,\ldots,x_{l_1(r)},\xi_{l_1(r)+1},\ldots,\xi_{l_1(r^{1/\theta})} ;y_1,\ldots,y_{l_2(r)})$ by $N$-th approximate squares $Q_{N, r^{1/\theta}}(\bar x, \bar y)$ (each with height and base length $r^{1/\theta}$),
where $(\bar x, \bar y)$ satisfies
\[
\begin{aligned}
\bar x_m=x_m & \text{ for } m=1,\ldots,l_1(r), \\
\bar x_m=\xi_m & \text{ for } m=l_1(r)+1,\ldots,l_1(r^{1/\theta}),\\
\bar y_m=y_m & \text{ for } 1\leq m\leq l_2(r),\\
\bar y_m\in\Omega'\vert_N & \text{ for }  l_2(r)+1\leq m\leq l_2(r^{1/\theta}).
\end{aligned}
\]
There are $|\Om'|_N|^{l_2(r^{1/\theta}) - l_2(r)}$ of such choices. Putting these estimates together, we have
    \begin{equation*}
    \begin{split}
        &{\tilde{N}}_{d_{\infty}^{N}}(Q_{N,r}(x,y),r^{1/\theta})\\
        &\asymp
        \left(\prod_{m=l_1(r)+1}^{l_1(r^{1/\theta})} |\pi_N^{-1}(y_m)|\right) \left(|\Omega'\vert_N|^{l_2(r^{1/\theta})-l_2(r)}\right) \\
        &\overset{(\star\star)}{\leq} \left(\sup_{v\in\Omega'\vert_N}|\pi_N^{-1}(v)|\right)^{l_1(r^{1/\theta})-l_1(r)} \left(|\Omega'\vert_N|^{l_2(r^{1/\theta})-l_2(r)}\right) \\
        &= \left(\sup_{v\in\Omega'\vert_N}|\pi_N^{-1}(v)|\right)^{\lfloor\log r/\log a\rfloor-\lfloor\log r/\theta\log a\rfloor} \left(|\Omega'\vert_N|^{\lfloor\log r/\log b\rfloor-\lfloor\log r/\theta\log b\rfloor}\right).
    \end{split}
\end{equation*}
Taking logarithms, dividing $N$ and taking $N\to\infty$, we obtain that
\begin{equation*}
    \begin{split}
        &\widetilde{S}(X_\Omega,r,r^{1/\theta})=\lim_{N\to\infty}\frac{1}{N}\sup_{(x, y) \in \left(\Omega|_N\right)^{\mathbb{N}}} \log {\tilde{N}}_{d_{\infty}^{N}}(Q_{N,r}(x,y),r^{1/\theta})\\
        &\leq \left(\frac{\log r}{\log a}-\frac{\log r}{\theta\log a}+1\right) \htop(\Omega\vert\Omega', \sigma) +\left(\frac{\log r}{\log b}-\frac{\log r}{\theta\log b}+1\right) \htop(\Omega',\sigma)\\
        &=\left(\frac{\htop(\Omega\vert\Omega', \sigma)}{\log a} + 
      \frac{\htop(\Omega', \sigma)}{\log b}\right) \log r^{1-1/\theta} +\htop(\Omega\vert\Omega', \sigma)+\htop(\Omega', \sigma)
    \end{split}
\end{equation*}
and therefore
\begin{equation*}
    \begin{split}
        \madim^{\theta}\left(X_\Omega, \sigma, d\right) \leq
        \frac{\htop(\Omega\vert\Omega', \sigma)}{\log a} + \frac{\htop(\Omega', \sigma)}{\log b}.
    \end{split}
\end{equation*}
This yields the desired upper bound. Observe that it we choose $(x, y) \in \left(\Omega|_N\right)^{\mathbb{N}}$ with $|\pi_N^{-1}(y_m)|=\sup_{v\in\Omega'\vert_N}|\pi_N^{-1}(v)|$ for $m=l_1(r)+1,\ldots,l_1(r^{1/\theta})$, then the inequality $(\star\star)$ in the above argument is replaced by equality and our covering estimates were optimal up to multiplicative constants. Therefore, in this way we will obtain the required lower bound for the mean Assouad spectrum of $(X_{\Omega},\sigma,d)$.
\end{proof}

\subsection{Relations with finite-dimensional carpets}\label{sec: proj adim BM}

While the Bedford-McMullen carpet systems are infinite-dimensional objects, they are related to the finite-dimensional self-affine carpets. The connection is easiest when $\Omega$ is a full shift.

\begin{example}
    Fix a non-empty subset $R\subset A\times B$. We define  a Bedford--McMullen carpet by
    \[  C_R = \left\{\left(\sum_{m=1}^\infty \frac{x_m}{a^m}, \sum_{m=1}^\infty \frac{y_m}{b^m}\right) 
      \in [0,1]^2 \middle|\, 
     (x_m, y_m)\in R \text{ for all $m \geq 1$}\right\}. \]
     Let $\widehat{d}$ be the metric on the plane defined by
     \[\widehat{d}\left((x, y), (x^\prime, y^\prime)\right)  = \max\left\{|x-x^\prime|, |y-y^\prime|\right\}. \]
     For each $v\in B$ we denote by $t(v)$ the number of $u\in A$ satisfying $(u,v)\in R$. Denote by $R'$ the collection of $v\in B$ for which there exists $u\in A$ with $(u,v)\in R$. The Assouad dimension of $C_R$ is given by
     \begin{equation}\label{Assouad dim of carpet}
         \adim(C_R,\widehat{d})=\frac{\log \max_{v\in B}t(v)}{\log a}+\frac{\log |R'|}{\log b},
     \end{equation}
     see e.g. \cite[Theorem 1.1]{Mac11}.
     We now explain how \eqref{eq: mean Assouad dimension of carpet systems} is related to \eqref{Assouad dim of carpet}. Let $\Omega=R^{\N}$. First, observe that $\Omega'=\pi(\Omega)=\pi(R^{\N})=(R')^{\N}$ and hence $\htop(\Omega', \sigma)=\lim_{N\to\infty}\frac{\log|\Omega'|_N|}{N}=\lim_{N\to\infty}\frac{\log|(R')^N|}{N}=\log |R'|$. Moreover, by Lemma \ref{formula:topological conditional entropy of subshifts},
     \begin{equation*}
     \begin{split}
         \htop(\Omega\vert\Omega', \sigma)&=\lim_{N\to\infty}\frac{\sup_{v\in \Omega'|_{N}}\log|\pi_{N}^{-1}(v)|}{N}\\
         &=\lim_{N\to\infty}\frac{\sup_{v=(v_1,\ldots,v_N)\in (R')^{N}}\log(t(v_1)\cdots t(v_N))}{N}\\
         &=\log\max_{v\in B}t(v).
     \end{split}
    \end{equation*}
    We conclude $\madim\left(X_\Omega, \sigma, d\right)  =\adim(C_R,\widehat{d})$. Indeed, if $\Omega=R^{\N}$ is a full shift, we have $X_{\Omega}=(C_R)^{\N}$, i.e., the carpet system is the full shift on the planar Bedford--McMullen carpet (as observed directly from the construction of $X_{\Omega}$), so $\madim\left(X_\Omega, \sigma, d\right)  =\adim(C_R,\widehat{d})$ by Proposition \ref{madim=adim}. 
\end{example}

For a general subshift $\Om \subset (A \times B)^\N$, the connection is less direct. For given $N\in\N$, it is easy to see that the set $X_{\Omega}|_N$ defined in \eqref{eq: X Omega N} is a self-affine subset of $[0,1]^N \times [0,1]^N$. We have the following:

\begin{prop}\label{prop: BM prodim}
    Let $(X_{\Omega},\sigma)$ be a carpet system equipped with distance \eqref{eq: metric on infinite dimensional carpets}. Then we have $\madim\left(X_\Omega, \sigma, d\right)=\lim \limits_{N \to \infty} \frac{1}{N} \adim (X_{\Omega}|_N)$.
\end{prop}
\begin{proof}
    Fix $N\in\N$. Combining \eqref{eq: l1 l2 def} and Claim \ref{upper bound of approximate squares}, there exists $C>0$ such that for all $(x,y)\in(\Omega|_N)^{\N}$ and $0<\rho<r$:
    \begin{equation*}
        \begin{split}
            {\tilde{N}}_{d_{\infty}^{N}}(Q_{N,r}(x,y),\rho) &\leq C\left(\sup_{v\in\Omega'\vert_N}|\pi_N^{-1}(v)|\right)^{l_1(\rho)-l_1(r)} \left(|\Omega'\vert_N|^{l_2(\rho)-l_2(r)}\right)\\
            &\leq C\left(\sup_{v\in\Omega'\vert_N}|\pi_N^{-1}(v)|\right)^{\log_a{\frac{r}{\rho}}+1} \left(|\Omega'\vert_N|^{\log_b{\frac{r}{\rho}}+1}\right)\\
            &=C_N\left(\frac{r}{\rho}\right)^{\log_a{\sup_{v\in\Omega'\vert_N}|\pi_N^{-1}(v)|}+\log_b{|\Omega'\vert_N|}},
        \end{split}
    \end{equation*}
    where $C_N=C\sup_{v\in\Omega'\vert_N}|\pi_N^{-1}(v)|\cdot|\Omega'\vert_N|$. We conclude $\adim (X_{\Omega}|_N)\leq \log_a{\sup_{v\in\Omega'\vert_N}|\pi_N^{-1}(v)|}+\log_b{|\Omega'\vert_N|}$.

    Similar to Subsection \ref{Sec:mean Assouad dimension: lower bound}, we may find a sequence of points $(x, y)\in (\Omega|_N)^{\N}$ and scales $0<\rho<r$ satisfying
    \begin{equation*}
    \begin{split}
        {\tilde{N}}_{d_{\infty}^{N}}(Q_{N,r}(x,y),\rho)&\asymp \left(\sup_{v\in\Omega'\vert_N}|\pi_N^{-1}(v)|\right)^{l_1(\rho)-l_1(r)} \left(|\Omega'\vert_N|^{l_2(\rho)-l_2(r)}\right)\\
        % &\geq \left(\sup_{v\in\Omega'\vert_N}|\pi_N^{-1}(v)|\right)^{\log_a{\frac{r}{\rho}}-1} \left(|\Omega'\vert_N|^{\log_b{\frac{r}{\rho}}-1}\right)\\
            &\geq C'_N\left(\frac{r}{\rho}\right)^{\log_a{\sup_{v\in\Omega'\vert_N}|\pi_N^{-1}(v)|}+\log_b{|\Omega'\vert_N|}},
    \end{split}
\end{equation*}
where $C'_N=(\sup_{v\in\Omega'\vert_N}|\pi_N^{-1}(v)|\cdot|\Omega'\vert_N|)^{-1}$. We conclude $\adim (X_{\Omega}|_N)= \log_a{\sup_{v\in\Omega'\vert_N}|\pi_N^{-1}(v)|}+\log_b{|\Omega'\vert_N|}$ and hence by Theorem \ref{theorem: mean Assouad dimension of carpet systems}
\begin{equation*}
    \begin{split}
        \lim \limits_{N \to \infty} \frac{\adim (X_{\Omega}|_N)}{N} & =\lim \limits_{N \to \infty} \frac{\log_a{\sup_{v\in\Omega'\vert_N}|\pi_N^{-1}(v)|}}{N} + \lim \limits_{N \to \infty} \frac{\log_b{|\Omega'\vert_N|}}{N}\\
        & =\frac{\htop(\Omega\vert\Omega', \sigma)}{\log a} + 
      \frac{\htop(\Omega', \sigma)}{\log b}=\madim\left(X_\Omega, \sigma, d\right) .
    \end{split}
\end{equation*}
\end{proof}

\begin{rem}
More generally, for a subshift $\mS \subset F^{\N}$ one can define the \textbf{projective Assouad dimension} as
\[ \pdim(\mS) = \lim \limits_{N \to \infty} \frac{\adim (\tau_N (\mS))}{N},\]
where $\tau_N: F^{\N}\to F^N$ the projection onto the first $N$-coordinates. The limit exists due to the subadditivity of the Assouad dimension: $\adim(A \times B) \leq \adim(A) + \adim(B)$. This is an Assouad dimension version of the projective dimension introduced by Gromov \cite{G}. It is not difficult to observe that $\madim(\mS, \sigma, d) \leq \pdim(\mS)$ holds and this inequality can be strict (for instance, this is the case for the subshift from example \ref{ex: mdim def compare}). Proposition \ref{prop: BM prodim} says that equality holds for the Bedford-McMullen carpet systems.
\end{rem}

 \section{On other possible definitions of mean Assouad dimension}\label{sec: other def}

In this section we discuss other possible definitions of mean Assouad dimension and compare them with the one adopted in this paper. We hope to convince the reader that our choice is reasonable.

First, we show that exchanging the order of $\lim \limits_{M \to \infty}$ and $\sup \limits_{x \in X}$ in the definition leads to the same quantity.

\begin{prop}\label{prop: madim alter def}
For $0 < \rho < r$ define

\[ \overline{S}(X, r, \rho) = \sup \limits_{x \in X} \limsup \limits_{M \to \infty} \frac{1}{M} \log N_{d_M}(B_{d_M}(x,r),\rho) \]
and
\[ \underline{S}(X, r, \rho) = \sup \limits_{x \in X} \liminf \limits_{M \to \infty} \frac{1}{M} \log N_{d_M}(B_{d_M}(x,r),\rho). \]
Then
\[
\begin{split}
\madim(X, T, d) & = \inf \left\{ s > 0 : \underset{C>0}{\exists}\ \underset{0 < \rho < r}{\forall}\ e^{\overline{S}(X, r, \rho)} \leq C(r /\rho)^s \right\} \\
& = \inf \left\{ s > 0 : \underset{C>0}{\exists}\ \underset{0 < \rho < r}{\forall}\ e^{\underline{S}(X, r, \rho)} \leq C(r /\rho)^s \right\}.
\end{split}
\]
\end{prop}

\begin{proof}
We shall prove the following inequality:
\begin{equation}\label{eq: S underS ineq}
S(X, r, \rho) \leq \underline{S}(X, 2r, \rho).
\end{equation}

As clearly $\underline{S}(X, r, \rho) \leq \overline{S}(X, r, \rho) \leq S(X, r, \rho)$, obtaining \eqref{eq: S underS ineq} will finish the proof.

For proving \eqref{eq: S underS ineq}, fix $\eps > 0$. By the definition of $S(X, r, \rho)$, for all $M$ large enough there exists $x_M \in X$ such that
\[ \frac{1}{M} \log N_{d_M}(B_{d_M}(x_M,r),\rho) \geq S(X, r, \rho) - \eps. \]
Take a convergent subsequence $x_{M_k} \to x^*$. Then for $k$ large enough to guarantee $d(x_{M_k}, x^*) < r$, we have
\[
S(X, r, \rho) \leq \frac{1}{M_k} \log N_{d_{M_k}}(B_{d_{M_k}}(x_{M_k},r),\rho) + \eps  \leq \frac{1}{M_k} \log N_{d_{M_k}}(B_{d_{M_k}}(x^*,2r),\rho) + \eps
\]
Taking $\liminf \limits_{k \to \infty}$, we obtain
\[
S(X, r, \rho) \leq \liminf \limits_{k \to \infty} \frac{1}{M_k} \log N_{d_{M_k}}(B_{d_{M_k}}(x^*,2r),\rho) + \eps,
\]
so
\[
S(X, r, \rho) \leq \sup \limits_{x \in X} \liminf \limits_{M \to \infty} \frac{1}{M} \log N_{d_{M}}(B_{d_{M}}(x,2r),\rho) + \eps = \underline{S}(X, 2r, \rho) + \eps.
\]
As $\eps$ is arbitrary, \eqref{eq: S underS ineq} is established.
\end{proof}

It is straightforward to extend Proposition \ref{prop: madim alter def} to the mean Assouad spectrum.

In our definition of $\madim(X,T,d)$, uniformity of the bound is required with respect to $0 < \rho < r$ and $x \in X$ (as in the definition of the Assouad dimension), but with those fixed one is allowed to take the limit with $M \to \infty$. A stronger requirement would be to require a uniform bound also with respect to $M$. This however leads to an invariant with rather undesirable properties from the dynamical point of view. To be more precise, define

\[
\begin{split}
\widehat{\madim}(X,T,d) = \inf \Big\{ s > 0 : & \text{ there exists } C>0 \text{ and } M(C) \in \N \text{ such that for all } M > M(C) \\
& \sup \limits_{x \in X} N_{d_M}(B_{d_M}(x, r), \rho) \leq C (r / \rho)^{sM} \text{ for all } r > 0 \text{ and } 0 < \rho < r  \Big\}.
\end{split}
\]

The following example shows two drawbacks of the above definition. First, $\widehat{\madim}(X,T,d) > \widehat{\madim}(\Om(T),T,d)$ might occur. Second, subshifts of $[0,1]^\Z$ might have infinite $\widehat{\madim}$.

\begin{example}\label{ex: mdim def compare}
We shall use \cite[Example IV-C.2]{mmdimcompress}. Endow $[0,1]^\Z$ with the metric $d(x,y) = \sum \limits_{i \in \Z} 2^{-|i|}|x_i - y_i|$, where $x = (x_i)_{i \in \Z},\ y = (y_i)_{i \in \Z}$.  Let $\sigma : [0,1]^\Z \to [0,1]^\Z$ be the left shift. For each $m \in \N$, let 
\[ A_m = \{ (x_i)_{i \in \Z} : 0 \leq x_i \leq 1/m \text{ if } 0 \leq i < m, \text{ and } x_i = 0 \text{ otherwise } \},  \]
\[X_m = \bigcup \limits_{n \in \Z} \sigma^n (A_m) \]
and set
\[ X = \bigcup \limits_{m=1}^\infty X_m.\]
It is easy to see that $X$ is a compact $\sigma$-invariant set, hence $(X, \sigma, d)$ is a topological dynamical system. Moreover, it has the property that $\lim \limits_{n \to \infty} \sigma^n x = \vec{0} = (\ldots, 0,0,0, \ldots)$ for every $x \in X$, hence $\Omega(\sigma) = \{ \vec{0} \}$. Therefore Proposition \ref{prop: nonwandering} gives that
\[ \madim(X, \sigma, d) = \madim(\Omega(\sigma), \sigma, d) = 0. \]
We claim that
\begin{equation}\label{eq: widehat madim 1}
\widehat{\madim}(X, \sigma, d) = \infty \text{ and } \widehat{\madim}(\Omega(\sigma),\sigma,d) = 0.
\end{equation}

The equality $\widehat{\madim}(\Omega(\sigma),\sigma,d) = 0$ follows directly from $\Omega(\sigma) = \{ \vec{0} \}$.

To see $\widehat{\madim}(X, \sigma, d) = \infty$ we argue as follows. As the shift map $\sigma$ is Lipschitz continuous in metric $d$, we see that for each $M\geq1$, the metric $d_M$ is bi-Lipschitz equivalent with $d$. Therefore
\begin{equation}\label{eq: adim d_M}
\adim(X,d) = \adim(X, d_M) = \infty \text{ for } M \geq 1.
\end{equation}
Suppose now $\widehat{\madim}(X, \sigma, d) < \infty$. Then, there exists $0 < s < \infty$ and $C>0$ such that for all $M$ large enough we have
\[\sup \limits_{x \in X} N_{d_M}(B_{d_M}(x, r), \rho) \leq C (r / \rho)^{sM} \text{ for all } r > 0 \text{ and } 0 < \rho < r.\]
This however implies $\adim(X,d_M)\leq sM$ for large enough $M$, a contradiction with \eqref{eq: adim d_M}.

\end{example}

\begin{rem}
Recently, an alternative definition of mean Assouad dimension was provided in \cite{LSL24}. In our notation the definition is as follows:

\[
\begin{split}
\widetilde{\madim}(X,T,d) = \inf \Big\{ a : & \text{ there exists } C>0 \text{ and } M(C) \in \N \text{ such that for all } M > M(C) \\
& \sup \limits_{x \in X} N_{d_M}(B_{d_M}(x, r), \rho) \leq C (r / \rho)^a \text{ for all } r > 0 \text{ and } 0 < \rho < \min\{1, r\}  \Big\}.
\end{split}
\]

This notion is however of a different flavour than the mean dimensions which we have discussed so far, as it has the property
\[\widetilde{\madim}(X,T,d) < \infty\ \Longrightarrow\ \htop(X,T,d) = 0\]
(as by taking $r > \diam(X)$ we in fact have $\lim \limits_{M \to \infty } \frac{1}{M} \log N_{d_M}(X,\rho) = 0$ for all $0 < \rho < 1$ provided that $\widetilde{\madim}(X,T,d) < \infty$).
\end{rem}

\bibliographystyle{alpha}
\bibliography{dynamical-spectra}

@book {Fra21,
    AUTHOR = {Fraser, Jonathan M.},
     TITLE = {Assouad dimension and fractal geometry},
    SERIES = {Cambridge Tracts in Mathematics},
    VOLUME = {222},
 PUBLISHER = {Cambridge University Press, Cambridge},
      YEAR = {2021},
     PAGES = {xvi+269},
      ISBN = {978-1-108-47865-6},
   MRCLASS = {28-02 (28A78 28A80)},
  MRNUMBER = {4411274},
MRREVIEWER = {Tushar Das},
       DOI = {10.1017/9781108778459},
       URL = {https://doi.org/10.1017/9781108778459},
}

@article {FY18,
    AUTHOR = {Fraser, Jonathan M. and Yu, Han},
     TITLE = {New dimension spectra: finer information on scaling and
              homogeneity},
   JOURNAL = {Adv. Math.},
  FJOURNAL = {Advances in Mathematics},
    VOLUME = {329},
      YEAR = {2018},
     PAGES = {273--328},
      ISSN = {0001-8708},
   MRCLASS = {28A80 (26A21 28A78 30L05)},
  MRNUMBER = {3783415},
MRREVIEWER = {Jeremy T. Tyson},
       DOI = {10.1016/j.aim.2017.12.019},
       URL = {https://doi.org/10.1016/j.aim.2017.12.019},
}

@article {FY18b,
    AUTHOR = {Fraser, Jonathan M. and Yu, Han},
     TITLE = {Assouad-type spectra for some fractal families},
   JOURNAL = {Indiana Univ. Math. J.},
  FJOURNAL = {Indiana University Mathematics Journal},
    VOLUME = {67},
      YEAR = {2018},
    NUMBER = {5},
     PAGES = {2005--2043},
      ISSN = {0022-2518},
   MRCLASS = {28A80 (60J80)},
  MRNUMBER = {3875249},
MRREVIEWER = {Sascha Troscheit},
       DOI = {10.1512/iumj.2018.67.7509},
       URL = {https://doi.org/10.1512/iumj.2018.67.7509},
}

@article {Mc84,
    AUTHOR = {McMullen, Curt},
     TITLE = {The {H}ausdorff dimension of general {S}ierpi\'{n}ski carpets},
   JOURNAL = {Nagoya Math. J.},
  FJOURNAL = {Nagoya Mathematical Journal},
    VOLUME = {96},
      YEAR = {1984},
     PAGES = {1--9},
      ISSN = {0027-7630},
   MRCLASS = {11K55},
  MRNUMBER = {771063},
MRREVIEWER = {A. D. Pollington},
       DOI = {10.1017/S0027763000021085},
       URL = {https://doi.org/10.1017/S0027763000021085},
}

@article {Mac11,
    AUTHOR = {Mackay, John M.},
     TITLE = {Assouad dimension of self-affine carpets},
   JOURNAL = {Conform. Geom. Dyn.},
  FJOURNAL = {Conformal Geometry and Dynamics. An Electronic Journal of the
              American Mathematical Society},
    VOLUME = {15},
      YEAR = {2011},
     PAGES = {177--187},
   MRCLASS = {37F35 (28A78)},
  MRNUMBER = {2846307},
MRREVIEWER = {Kevin Wildrick},
       DOI = {10.1090/S1088-4173-2011-00232-3},
       URL = {https://doi.org/10.1090/S1088-4173-2011-00232-3},
}

@article {Fra14,
    AUTHOR = {Fraser, Jonathan M.},
     TITLE = {Assouad type dimensions and homogeneity of fractals},
   JOURNAL = {Trans. Amer. Math. Soc.},
  FJOURNAL = {Transactions of the American Mathematical Society},
    VOLUME = {366},
      YEAR = {2014},
    NUMBER = {12},
     PAGES = {6687--6733},
      ISSN = {0002-9947},
   MRCLASS = {28A80 (28A20 28A78 28C15)},
  MRNUMBER = {3267023},
MRREVIEWER = {Takahisa Miyata},
       DOI = {10.1090/S0002-9947-2014-06202-8},
       URL = {https://doi.org/10.1090/S0002-9947-2014-06202-8},
}

@phdthesis {Bed84,
    AUTHOR = {Bedford, Tim},
     TITLE = {Crinkly curves, Markov partitions and box dimensions in self-similar sets},
      NOTE = {University of Warwick},
      YEAR = {1984},
   MRCLASS = {Thesis},
}

@article {BK24,
    AUTHOR = {Banaji, Amlan and Kolossv\'{a}ry, Istv\'{a}n},
     TITLE = {Intermediate dimensions of {B}edford-{M}c{M}ullen carpets with
              applications to {L}ipschitz equivalence},
   JOURNAL = {Adv. Math.},
  FJOURNAL = {Advances in Mathematics},
    VOLUME = {449},
      YEAR = {2024},
     PAGES = {Paper No. 109735, 69},
      ISSN = {0001-8708},
   MRCLASS = {28A80 (28A78 37C45)},
  MRNUMBER = {4751541},
       DOI = {10.1016/j.aim.2024.109735},
       URL = {https://doi.org/10.1016/j.aim.2024.109735},
}

@incollection {Fra21b,
    AUTHOR = {Fraser, Jonathan M.},
     TITLE = {Interpolating between dimensions},
 BOOKTITLE = {Fractal geometry and stochastics {VI}},
    SERIES = {Progr. Probab.},
    VOLUME = {76},
     PAGES = {3--24},
 PUBLISHER = {Birkh\"{a}user/Springer, Cham},
      YEAR = {2021},
   MRCLASS = {28A80},
  MRNUMBER = {4237247},
       DOI = {10.1007/978-3-030-59649-1\_1},
       URL = {https://doi.org/10.1007/978-3-030-59649-1_1},
}

@article {BFKR24,
    AUTHOR = {Banaji, Amlan and Fraser, Jonathan M. and Kolossv\'{a}ry,
              Istv\'{a}n and Rutar, Alex},
     TITLE = {Assouad spectrum of {G}atzouras--{L}alley carpets},
   JOURNAL = {Adv. Math.},
  FJOURNAL = {Advances in Mathematics},
    VOLUME = {484},
      YEAR = {2026},
     PAGES = {Paper No. 110707},
      ISSN = {0001-8708,1090-2082},
   MRCLASS = {28A80 (37C45 49N15)},
  MRNUMBER = {4993399},
       DOI = {10.1016/j.aim.2025.110707},
       URL = {https://doi.org/10.1016/j.aim.2025.110707},
}

@misc{Nar24,
      title={{URP}, comparison, mean dimension, and sharp shift embeddability}, 
      author={Petr Naryshkin},
      url={https://arxiv.org/abs/2410.01757v2}, 
      eprint={2410.01757},
      year={2024},
      archivePrefix={arXiv},
      primaryClass={math.DS},
      HOWPUBLISHED={Preprint, available at https://arxiv.org/abs/2410.01757v2}, 
}

@book {Rob11,
    AUTHOR = {Robinson, James C.},
     TITLE = {Dimensions, embeddings, and attractors},
    SERIES = {Cambridge Tracts in Mathematics},
    VOLUME = {186},
 PUBLISHER = {Cambridge University Press, Cambridge},
      YEAR = {2011},
     PAGES = {xii+205},
      ISBN = {978-0-521-89805-8},
   MRCLASS = {37C45 (28A78 28A80 37-02 37C70 37L05 37L30)},
  MRNUMBER = {2767108},
MRREVIEWER = {Vaughn Climenhaga},
}

@book{Coo15,
  title={Topological dimension and dynamical systems},
  author={Coornaert, Michel},
  year={2015},
  publisher={Springer}
}

@article {Tsu22,
    AUTHOR = {Tsukamoto, Masaki},
     TITLE = {Remark on the local nature of metric mean dimension},
   JOURNAL = {Kyushu J. Math.},
  FJOURNAL = {Kyushu Journal of Mathematics},
    VOLUME = {76},
      YEAR = {2022},
    NUMBER = {1},
     PAGES = {143--162},
      ISSN = {1340-6116},
   MRCLASS = {37B40 (32S45 37B52 37C45)},
  MRNUMBER = {4414307},
MRREVIEWER = {Bingbing Liang},
       DOI = {10.2206/kyushujm.76.143},
       URL = {https://doi.org/10.2206/kyushujm.76.143},
}

@article {Shi22,
    AUTHOR = {Shi, Ruxi},
     TITLE = {On variational principles for metric mean dimension},
   JOURNAL = {IEEE Trans. Inform. Theory},
  FJOURNAL = {Institute of Electrical and Electronics Engineers.
              Transactions on Information Theory},
    VOLUME = {68},
      YEAR = {2022},
    NUMBER = {7},
     PAGES = {4282--4288},
      ISSN = {0018-9448},
   MRCLASS = {37D35 (94A17)},
  MRNUMBER = {4449040},
MRREVIEWER = {Xinsheng Wang},
       DOI = {10.1109/tit.2022.3157786},
       URL = {https://doi.org/10.1109/tit.2022.3157786},
}

@misc{VV17,
      title={Rate distortion theory, metric mean dimension and measure theoretic entropy}, 
      author={Anibal Velozo and Renato Velozo},
      eprint={1707.05762},
      year={2017},
      archivePrefix={arXiv},
      primaryClass={math.DS},
      HOWPUBLISHED={Preprint, available at https://arxiv.org/abs/1707.05762}, 
}

@misc{Levin23,
      title={Finite-to-one equivariant maps and mean dimension}, 
      author={Michael Levin},
      eprint={2312.04689},
      year={2023},
      archivePrefix={arXiv},
      primaryClass={math.DS},
      HOWPUBLISHED={Preprint, available at https://arxiv.org/abs/2312.04689}, 
}

@article {mmdimcompress,
	AUTHOR = {Gutman, Yonatan and {\'{S}}piewak, Adam},
	TITLE = {Metric mean dimension and analog compression},
	JOURNAL = {IEEE Trans. Inform. Theory},
	FJOURNAL = {Institute of Electrical and Electronics Engineers.
	Transactions on Information Theory},
	VOLUME = {66},
	YEAR = {2020},
	NUMBER = {11},
	PAGES = {6977--6998},
	ISSN = {0018-9448},
	MRCLASS = {94A12 (94A17)},
	MRNUMBER = {4173622},
	MRREVIEWER = {Fei Chen},
	DOI = {10.1109/TIT.2020.2992388},
	URL = {https://doi.org/10.1109/TIT.2020.2992388},
}

@article {G,
    AUTHOR = {Gromov, Misha},
     TITLE = {Topological invariants of dynamical systems and spaces of
              holomorphic maps. {I}},
   JOURNAL = {Math. Phys. Anal. Geom.},
  FJOURNAL = {Mathematical Physics, Analysis and Geometry. An International
              Journal Devoted to the Theory and Applications of Analysis and
              Geometry to Physics},
    VOLUME = {2},
      YEAR = {1999},
    NUMBER = {4},
     PAGES = {323--415},
      ISSN = {1385-0172},
     CODEN = {MPAGFO},
   MRCLASS = {37B99 (32H02 53C23 58E20)},
  MRNUMBER = {MR1742309 (2001j:37037)},
MRREVIEWER = {Boris Hasselblatt},
}

@book {Dow11,
    AUTHOR = {Downarowicz, Tomasz},
     TITLE = {Entropy in dynamical systems},
    SERIES = {New Mathematical Monographs},
    VOLUME = {18},
 PUBLISHER = {Cambridge University Press, Cambridge},
      YEAR = {2011},
     PAGES = {xii+391},
      ISBN = {978-0-521-88885-1},
   MRCLASS = {37-02 (28D05 37A35 37B40 54H20 94A17)},
  MRNUMBER = {2809170},
MRREVIEWER = {Michael Hochman},
       DOI = {10.1017/CBO9780511976155},
       URL = {https://doi.org/10.1017/CBO9780511976155},
}

@article {FH16,
    AUTHOR = {Feng, De-Jun and Huang, Wen},
     TITLE = {Variational principle for weighted topological pressure},
   JOURNAL = {J. Math. Pures Appl. (9)},
  FJOURNAL = {Journal de Math\'{e}matiques Pures et Appliqu\'{e}es. Neuvi\`eme S\'{e}rie},
    VOLUME = {106},
      YEAR = {2016},
    NUMBER = {3},
     PAGES = {411--452},
      ISSN = {0021-7824},
   MRCLASS = {37D35 (37B40 37C45)},
  MRNUMBER = {3520443},
MRREVIEWER = {Zhiming Li},
       DOI = {10.1016/j.matpur.2016.02.016},
       URL = {https://doi.org/10.1016/j.matpur.2016.02.016},
}

@article {Tsuweighted,
    AUTHOR = {Tsukamoto, Masaki},
     TITLE = {New approach to weighted topological entropy and pressure},
   JOURNAL = {Ergodic Theory Dynam. Systems},
  FJOURNAL = {Ergodic Theory and Dynamical Systems},
    VOLUME = {43},
      YEAR = {2023},
    NUMBER = {3},
     PAGES = {1004--1034},
      ISSN = {0143-3857},
   MRCLASS = {37D35 (37A35 37B40 37C45)},
  MRNUMBER = {4544152},
       DOI = {10.1017/etds.2021.173},
       URL = {https://doi.org/10.1017/etds.2021.173},
}

@article {Mis76conditional,
    AUTHOR = {Micha\l  Misiurewicz},
     TITLE = {Topological conditional entropy},
   JOURNAL = {Studia Math.},
  FJOURNAL = {Polska Akademia Nauk. Instytut Matematyczny. Studia
              Mathematica},
    VOLUME = {55},
      YEAR = {1976},
    NUMBER = {2},
     PAGES = {175--200},
      ISSN = {0039-3223},
   MRCLASS = {54H20},
  MRNUMBER = {415587},
MRREVIEWER = {Timothy N.t Goodman},
       DOI = {10.4064/sm-55-2-175-200},
       URL = {https://doi.org/10.4064/sm-55-2-175-200},
}

@article {DS02fiber,
    AUTHOR = {Downarowicz, Tomasz and Serafin, Jacek},
     TITLE = {Fiber entropy and conditional variational principles in
              compact non-metrizable spaces},
   JOURNAL = {Fund. Math.},
  FJOURNAL = {Fundamenta Mathematicae},
    VOLUME = {172},
      YEAR = {2002},
    NUMBER = {3},
     PAGES = {217--247},
      ISSN = {0016-2736},
   MRCLASS = {37B40 (37A35)},
  MRNUMBER = {1898686},
MRREVIEWER = {Doris Fiebig},
       DOI = {10.4064/fm172-3-2},
       URL = {https://doi.org/10.4064/fm172-3-2},
}

@article {Tsu25carpets,
    AUTHOR = {Tsukamoto, Masaki},
     TITLE = {Mean {H}ausdorff dimension of some infinite-dimensional
              fractals},
   JOURNAL = {J. Anal. Math.},
  FJOURNAL = {Journal d'Analyse Math\'{e}matique},
    VOLUME = {155},
      YEAR = {2025},
    NUMBER = {1},
     PAGES = {235--286},
      ISSN = {0021-7670},
   MRCLASS = {28A80 (28A78)},
  MRNUMBER = {4905311},
       DOI = {10.1007/s11854-024-0353-0},
       URL = {https://doi.org/10.1007/s11854-024-0353-0},
}

@article {LW00,
    AUTHOR = {Lindenstrauss, Elon and Weiss, Benjamin},
     TITLE = {Mean topological dimension},
   JOURNAL = {Israel J. Math.},
  FJOURNAL = {Israel Journal of Mathematics},
    VOLUME = {115},
      YEAR = {2000},
     PAGES = {1--24},
      ISSN = {0021-2172},
   MRCLASS = {37B99 (37B05 37B40)},
  MRNUMBER = {1749670},
MRREVIEWER = {Kathleen M. Madden},
       DOI = {10.1007/BF02810577},
       URL = {https://doi.org/10.1007/BF02810577},
}

@article {Jur23,
    AUTHOR = {Jurga, Natalia},
     TITLE = {Nonexistence of the box dimension for dynamically invariant
              sets},
   JOURNAL = {Anal. PDE},
  FJOURNAL = {Analysis \& PDE},
    VOLUME = {16},
      YEAR = {2023},
    NUMBER = {10},
     PAGES = {2385--2399},
      ISSN = {2157-5045},
   MRCLASS = {37C45 (28A80 37D20)},
  MRNUMBER = {4678144},
MRREVIEWER = {Jonathan MacDonald Fraser},
       DOI = {10.2140/apde.2023.16.2385},
       URL = {https://doi.org/10.2140/apde.2023.16.2385},
}

@article {BR22,
    AUTHOR = {Banaji, Amlan and Rutar, Alex},
     TITLE = {Attainable forms of intermediate dimensions},
   JOURNAL = {Ann. Fenn. Math.},
  FJOURNAL = {Annales Fennici Mathematici},
    VOLUME = {47},
      YEAR = {2022},
    NUMBER = {2},
     PAGES = {939--960},
      ISSN = {2737-0690},
   MRCLASS = {28A78 (28A80)},
  MRNUMBER = {4448745},
MRREVIEWER = {Elena Hadzieva},
       DOI = {10.54330/afm.120529},
       URL = {https://doi.org/10.54330/afm.120529},
}

@article {GHM21,
    AUTHOR = {Garc\'{\i}a, Ignacio and Hare, Kathryn E. and Mendivil, Franklin},
     TITLE = {Intermediate {A}ssouad-like dimensions},
   JOURNAL = {J. Fractal Geom.},
  FJOURNAL = {Journal of Fractal Geometry. Mathematics of Fractals and
              Related Topics},
    VOLUME = {8},
      YEAR = {2021},
    NUMBER = {3},
     PAGES = {201--245},
      ISSN = {2308-1309},
   MRCLASS = {28A78 (28A80)},
  MRNUMBER = {4321216},
MRREVIEWER = {Jun Jie Miao},
       DOI = {10.4171/jfg/102},
       URL = {https://doi.org/10.4171/jfg/102},
}

@article {Feng24,
    AUTHOR = {Feng, Zhou},
     TITLE = {Intermediate dimensions under self-affine codings},
   JOURNAL = {Math. Z.},
  FJOURNAL = {Mathematische Zeitschrift},
    VOLUME = {307},
      YEAR = {2024},
    NUMBER = {1},
     PAGES = {Paper No. 21, 29},
      ISSN = {0025-5874},
   MRCLASS = {28A80 (31B15 60G22)},
  MRNUMBER = {4739520},
MRREVIEWER = {Amlan Banaji},
       DOI = {10.1007/s00209-024-03490-z},
       URL = {https://doi.org/10.1007/s00209-024-03490-z},
}

@article {BF23,
    AUTHOR = {Banaji, Amlan and Fraser, Jonathan M.},
     TITLE = {Intermediate dimensions of infinitely generated attractors},
   JOURNAL = {Trans. Amer. Math. Soc.},
  FJOURNAL = {Transactions of the American Mathematical Society},
    VOLUME = {376},
      YEAR = {2023},
    NUMBER = {4},
     PAGES = {2449--2479},
      ISSN = {0002-9947},
   MRCLASS = {28A80 (11K50 37B10)},
  MRNUMBER = {4557871},
MRREVIEWER = {Manuel Mor\'{a}n},
       DOI = {10.1090/tran/8766},
       URL = {https://doi.org/10.1090/tran/8766},
}

@article {Rut24,
    AUTHOR = {Rutar, Alex},
     TITLE = {Attainable forms of {A}ssouad spectra},
   JOURNAL = {Indiana Univ. Math. J.},
  FJOURNAL = {Indiana University Mathematics Journal},
    VOLUME = {73},
      YEAR = {2024},
    NUMBER = {4},
     PAGES = {1331--1356},
      ISSN = {0022-2518},
   MRCLASS = {28A78},
  MRNUMBER = {4806752},
}

@article {BF24,
    AUTHOR = {Banaji, Amlan and Fraser, Jonathan M.},
     TITLE = {Assouad type dimensions of infinitely generated self-conformal
              sets},
   JOURNAL = {Nonlinearity},
  FJOURNAL = {Nonlinearity},
    VOLUME = {37},
      YEAR = {2024},
    NUMBER = {4},
     PAGES = {Paper No. 045004, 32},
      ISSN = {0951-7715},
   MRCLASS = {28A80 (11K50)},
  MRNUMBER = {4714081},
MRREVIEWER = {Alex Rutar},
       DOI = {10.1088/1361-6544/ad2864},
       URL = {https://doi.org/10.1088/1361-6544/ad2864},
}

@article {BC23,
    AUTHOR = {Banaji, Amlan and Chen, Haipeng},
     TITLE = {Dimensions of popcorn-like pyramid sets},
   JOURNAL = {J. Fractal Geom.},
  FJOURNAL = {Journal of Fractal Geometry. Mathematics of Fractals and
              Related Topics},
    VOLUME = {10},
      YEAR = {2023},
    NUMBER = {1-2},
     PAGES = {151--168},
      ISSN = {2308-1309},
   MRCLASS = {28A80 (11B57)},
  MRNUMBER = {4627104},
       DOI = {10.4171/jfg/135},
       URL = {https://doi.org/10.4171/jfg/135},
}

@article {CFY22,
    AUTHOR = {Chen, Haipeng and Fraser, Jonathan M. and Yu, Han},
     TITLE = {Dimensions of the popcorn graph},
   JOURNAL = {Proc. Amer. Math. Soc.},
  FJOURNAL = {Proceedings of the American Mathematical Society},
    VOLUME = {150},
      YEAR = {2022},
    NUMBER = {11},
     PAGES = {4729--4742},
      ISSN = {0002-9939},
   MRCLASS = {28A80 (11B57)},
  MRNUMBER = {4489308},
MRREVIEWER = {Tomas Persson},
       DOI = {10.1090/proc/15729},
       URL = {https://doi.org/10.1090/proc/15729},
}

@article {Lin99,
    AUTHOR = {Lindenstrauss, Elon},
     TITLE = {Mean dimension, small entropy factors and an embedding
              theorem},
   JOURNAL = {Inst. Hautes \'{E}tudes Sci. Publ. Math.},
  FJOURNAL = {Institut des Hautes \'{E}tudes Scientifiques. Publications
              Math\'{e}matiques},
    NUMBER = {89},
      YEAR = {1999},
     PAGES = {227--262 (2000)},
      ISSN = {0073-8301},
   MRCLASS = {37B40 (37A35 54F45)},
  MRNUMBER = {1793417},
MRREVIEWER = {Mahendra G. Nadkarni},
       URL = {http://www.numdam.org/item?id=PMIHES_1999__89__227_0},
}

@article {Gut15,
    AUTHOR = {Gutman, Yonatan},
     TITLE = {Mean dimension and {J}aworski-type theorems},
   JOURNAL = {Proc. Lond. Math. Soc. (3)},
  FJOURNAL = {Proceedings of the London Mathematical Society. Third Series},
    VOLUME = {111},
      YEAR = {2015},
    NUMBER = {4},
     PAGES = {831--850},
      ISSN = {0024-6115},
   MRCLASS = {37A35 (37B10 54F45)},
  MRNUMBER = {3407186},
MRREVIEWER = {Vaughn Climenhaga},
       DOI = {10.1112/plms/pdv043},
       URL = {https://doi.org/10.1112/plms/pdv043},
}

@article {GT20,
    AUTHOR = {Gutman, Yonatan and Tsukamoto, Masaki},
     TITLE = {Embedding minimal dynamical systems into {H}ilbert cubes},
   JOURNAL = {Invent. Math.},
  FJOURNAL = {Inventiones Mathematicae},
    VOLUME = {221},
      YEAR = {2020},
    NUMBER = {1},
     PAGES = {113--166},
      ISSN = {0020-9910},
   MRCLASS = {37B05 (37A46 54F45 94A12)},
  MRNUMBER = {4105086},
MRREVIEWER = {Ryo Moore},
       DOI = {10.1007/s00222-019-00942-w},
       URL = {https://doi.org/10.1007/s00222-019-00942-w},
}

@article {GLT16,
    AUTHOR = {Gutman, Yonatan and Lindenstrauss, Elon and Tsukamoto, Masaki},
     TITLE = {Mean dimension of {$\Bbb{Z}^k$}-actions},
   JOURNAL = {Geom. Funct. Anal.},
  FJOURNAL = {Geometric and Functional Analysis},
    VOLUME = {26},
      YEAR = {2016},
    NUMBER = {3},
     PAGES = {778--817},
      ISSN = {1016-443X},
   MRCLASS = {37B40 (54F45)},
  MRNUMBER = {3540453},
MRREVIEWER = {Tom Meyerovitch},
       DOI = {10.1007/s00039-016-0372-9},
       URL = {https://doi.org/10.1007/s00039-016-0372-9},
}

@article {LT18,
    AUTHOR = {Lindenstrauss, Elon and Tsukamoto, Masaki},
     TITLE = {From rate distortion theory to metric mean dimension:
              variational principle},
   JOURNAL = {IEEE Trans. Inform. Theory},
  FJOURNAL = {Institute of Electrical and Electronics Engineers.
              Transactions on Information Theory},
    VOLUME = {64},
      YEAR = {2018},
    NUMBER = {5},
     PAGES = {3590--3609},
      ISSN = {0018-9448},
   MRCLASS = {94A17 (37C45)},
  MRNUMBER = {3798396},
MRREVIEWER = {Thomas Ward},
       DOI = {10.1109/TIT.2018.2806219},
       URL = {https://doi.org/10.1109/TIT.2018.2806219},
}

@article {LT19,
    AUTHOR = {Lindenstrauss, Elon and Tsukamoto, Masaki},
     TITLE = {Double variational principle for mean dimension},
   JOURNAL = {Geom. Funct. Anal.},
  FJOURNAL = {Geometric and Functional Analysis},
    VOLUME = {29},
      YEAR = {2019},
    NUMBER = {4},
     PAGES = {1048--1109},
      ISSN = {1016-443X},
   MRCLASS = {37C45 (28A75 94A34)},
  MRNUMBER = {3990194},
MRREVIEWER = {Tushar Das},
       DOI = {10.1007/s00039-019-00501-8},
       URL = {https://doi.org/10.1007/s00039-019-00501-8},
}

@misc{LSL24,
      title={On the mean ${\Psi}$-intermediate dimensions}, 
      author={Yu Liu and Bilel Selmi and Zhiming Li},
      url={https://arxiv.org/abs/2407.09843}, 
      eprint={2407.09843},
      year={2024},
      archivePrefix={arXiv},
      primaryClass={math.DS},
      HOWPUBLISHED={Preprint, available at https://arxiv.org/abs/2407.09843}, 
}

@article {Ban23,
    AUTHOR = {Banaji, Amlan},
     TITLE = {Generalised intermediate dimensions},
   JOURNAL = {Monatsh. Math.},
  FJOURNAL = {Monatshefte f\"{u}r Mathematik},
    VOLUME = {202},
      YEAR = {2023},
    NUMBER = {3},
     PAGES = {465--506},
      ISSN = {0026-9255},
   MRCLASS = {28A80 (28A78)},
  MRNUMBER = {4651960},
MRREVIEWER = {Alex Rutar},
       DOI = {10.1007/s00605-023-01884-5},
       URL = {https://doi.org/10.1007/s00605-023-01884-5},
}

@article {GQT19,
    AUTHOR = {Gutman, Yonatan and Qiao, Yixiao and Tsukamoto, Masaki},
     TITLE = {Application of signal analysis to the embedding problem of
              {$\Bbb Z^k$}-actions},
   JOURNAL = {Geom. Funct. Anal.},
  FJOURNAL = {Geometric and Functional Analysis},
    VOLUME = {29},
      YEAR = {2019},
    NUMBER = {5},
     PAGES = {1440--1502},
      ISSN = {1016-443X},
   MRCLASS = {37B52 (94A12)},
  MRNUMBER = {4025517},
MRREVIEWER = {Vladimir Garc\'{\i}a-Morales},
       DOI = {10.1007/s00039-019-00499-z},
       URL = {https://doi.org/10.1007/s00039-019-00499-z},
}

@article {Tsu25RandomBrody,
    AUTHOR = {Tsukamoto, Masaki},
     TITLE = {Rate distortion dimension of random {B}rody curves},
   JOURNAL = {Geom. Funct. Anal.},
  FJOURNAL = {Geometric and Functional Analysis},
    VOLUME = {35},
      YEAR = {2025},
    NUMBER = {3},
     PAGES = {915--978},
      ISSN = {1016-443X},
   MRCLASS = {32H30 (32Q45 37D15 37D35)},
  MRNUMBER = {4915746},
       DOI = {10.1007/s00039-025-00709-x},
       URL = {https://doi.org/10.1007/s00039-025-00709-x},
}

@article {Yang-Mills,
    AUTHOR = {Tsukamoto, Masaki},
     TITLE = {Large dynamics of {Y}ang-{M}ills theory: mean dimension
              formula},
   JOURNAL = {J. Anal. Math.},
  FJOURNAL = {Journal d'Analyse Math\'{e}matique},
    VOLUME = {134},
      YEAR = {2018},
    NUMBER = {2},
     PAGES = {455--499},
      ISSN = {0021-7670},
   MRCLASS = {58E15 (37L30 53C05)},
  MRNUMBER = {3771488},
MRREVIEWER = {W.-H. Steeb},
       DOI = {10.1007/s11854-018-0014-2},
       URL = {https://doi.org/10.1007/s11854-018-0014-2},
}

@article {Tsu18Brody,
    AUTHOR = {Tsukamoto, Masaki},
     TITLE = {Mean dimension of the dynamical system of {B}rody curves},
   JOURNAL = {Invent. Math.},
  FJOURNAL = {Inventiones Mathematicae},
    VOLUME = {211},
      YEAR = {2018},
    NUMBER = {3},
     PAGES = {935--968},
      ISSN = {0020-9910},
   MRCLASS = {32H30 (32Q45 37L99 54H20)},
  MRNUMBER = {3763403},
MRREVIEWER = {Tuan Dinh Huynh},
       DOI = {10.1007/s00222-017-0758-9},
       URL = {https://doi.org/10.1007/s00222-017-0758-9},
}

@article{RANonaut22,
 author = {Rodrigues, Fagner B. and Acevedo, Jeovanny Muentes},
 title = {Mean dimension and metric mean dimension for non-autonomous dynamical systems},
 fjournal = {Journal of Dynamical and Control Systems},
 journal = {J. Dyn. Control Syst.},
 issn = {1079-2724},
 volume = {28},
 number = {4},
 pages = {697--723},
 year = {2022},
 language = {English},
 doi = {10.1007/s10883-021-09541-6},
 zbMATH = {7593624},
 Zbl = {1506.37025}
}

@article{Gutman17,
 author = {Gutman, Yonatan},
 title = {Embedding topological dynamical systems with periodic points in cubical shifts},
 fjournal = {Ergodic Theory and Dynamical Systems},
 journal = {Ergodic Theory Dyn. Syst.},
 issn = {0143-3857},
 volume = {37},
 number = {2},
 pages = {512--538},
 year = {2017},
 language = {English},
 doi = {10.1017/etds.2015.40},
 url = {semanticscholar.org/paper/ef6558724b1dcaf6f6338cec1cb0038f1e809c7b},
 zbMATH = {6704321},
 Zbl = {1435.37034}
}

@misc{BowenTopA70,
 author = {Bowen, Rufus},
 title = {Topological entropy and axiom {{\(A\)}}},
 year = {1970},
 language = {English},
 howpublished = {Global {Analysis}, {Proc}. {Sympos}. {Pure} {Math}. 14, 23--41},
 zbMATH = {3330657},
 Zbl = {0207.54402}
}

@article{KSNonaut96,
 author = {Kolyada, Sergi{\u{\i}} and Snoha, {\v{L}}ubom{\'{\i}}r},
 title = {Topological entropy of nonautonomous dynamical systems},
 fjournal = {Random \& Computational Dynamics},
 journal = {Random Comput. Dyn.},
 issn = {1061-835X},
 volume = {4},
 number = {2-3},
 pages = {205--233},
 year = {1996},
 language = {English},
 zbMATH = {1026221},
 Zbl = {0909.54012}
}

@article {ENDuke,
    AUTHOR = {Elliott, George A. and Niu, Zhuang},
     TITLE = {The {C{$^*$}}-algebra of a minimal homeomorphism of zero mean
              dimension},
   JOURNAL = {Duke Math. J.},
  FJOURNAL = {Duke Mathematical Journal},
    VOLUME = {166},
      YEAR = {2017},
    NUMBER = {18},
     PAGES = {3569--3594},
      ISSN = {0012-7094},
   MRCLASS = {46L35 (37B05 46L85)},
  MRNUMBER = {3732883},
MRREVIEWER = {N. N. Ganikhodjaev},
       DOI = {10.1215/00127094-2017-0033},
       URL = {https://doi.org/10.1215/00127094-2017-0033},
}

@article {Niucomparisonradius,
    AUTHOR = {Niu, Zhuang},
     TITLE = {Comparison radius and mean topological dimension: {R}okhlin
              property, comparison of open sets, and subhomogeneous {$\rm
              C^*$}-algebras},
   JOURNAL = {J. Anal. Math.},
  FJOURNAL = {Journal d'Analyse Math\'{e}matique},
    VOLUME = {146},
      YEAR = {2022},
    NUMBER = {2},
     PAGES = {595--672},
      ISSN = {0021-7670},
   MRCLASS = {37A55},
  MRNUMBER = {4468009},
       DOI = {10.1007/s11854-022-0205-8},
       URL = {https://doi.org/10.1007/s11854-022-0205-8},
}

@article{BowenExpansive72,
 author = {Bowen, Rufus},
 title = {Entropy-expansive maps},
 fjournal = {Transactions of the American Mathematical Society},
 journal = {Trans. Am. Math. Soc.},
 issn = {0002-9947},
 volume = {164},
 pages = {323--331},
 year = {1972},
 language = {English},
 doi = {10.2307/1995978},
 zbMATH = {3363922},
 Zbl = {0229.28011}
}

@book{RudinFunctional,
 author = {Rudin, Walter},
 title = {Functional analysis.},
 edition = {2nd ed.},
 isbn = {0-07-054236-8},
 year = {1991},
 publisher = {New York, NY: McGraw-Hill},
 language = {English},
 zbMATH = {1022519},
 Zbl = {0867.46001}
}

@article{CPV24,
title = {A convex analysis approach to the metric mean dimension: Limits of scaled pressures and variational principles},
journal = {Advances in Mathematics},
volume = {436},
pages = {109407},
year = {2024},
issn = {0001-8708},
doi = {https://doi.org/10.1016/j.aim.2023.109407},
url = {https://www.sciencedirect.com/science/article/pii/S0001870823005509},
author = {Maria Carvalho and Gustavo Pessil and Paulo Varandas}
}

@misc{HuoSponges24,
      title={Mean dimension theory for infinite dimensional {B}edford--{M}cMullen sponges}, 
      author={Qiang Huo},
      year={2024},
      eprint={2412.02278},
      archivePrefix={arXiv},
      primaryClass={math.DS},
      HOWPUBLISHED={Preprint, available at https://arxiv.org/abs/2412.02278},
}

@misc{OR25branching,
      title={Two-scale branching functions and inhomogeneous attractors}, 
      author={Vilma Orgov\'{a}nyi and Alex Rutar},
      year={2025},
      eprint={2510.07013},
      archivePrefix={arXiv},
      primaryClass={math.DS},
      HOWPUBLISHED={Preprint, available at https://arxiv.org/abs/2510.07013},
}

@article{Olson02,
 author = {Olson, Eric},
 title = {Bouligand dimension and almost {Lipschitz} embeddings.},
 fjournal = {Pacific Journal of Mathematics},
 journal = {Pac. J. Math.},
 issn = {1945-5844},
 volume = {202},
 number = {2},
 pages = {459--474},
 year = {2002},
 language = {English},
 doi = {10.2140/pjm.2002.202.459},
 zbMATH = {1818954},
 Zbl = {1046.37013}
}

@article {LL25,
    AUTHOR = {Li, Xianqiang and Luo, Xiaofang},
     TITLE = {Mean {H}ausdorff dimension of some infinite dimensional
              fractals for amenable group actions},
   JOURNAL = {J. Math. Phys.},
  FJOURNAL = {Journal of Mathematical Physics},
    VOLUME = {66},
      YEAR = {2025},
    NUMBER = {12},
     PAGES = {Paper No. 122704},
      ISSN = {0022-2488,1089-7658},
   MRCLASS = {37B40 (28A78 28A80 54)},
  MRNUMBER = {4997027},
       DOI = {10.1063/5.0258258},
       URL = {https://doi.org/10.1063/5.0258258},
}

@article {LedWal77,
    AUTHOR = {Ledrappier, Fran{\c{c}}ois and Walters, Peter},
     TITLE = {A relativised variational principle for continuous
              transformations},
   JOURNAL = {J. London Math. Soc. (2)},
  FJOURNAL = {Journal of the London Mathematical Society. Second Series},
    VOLUME = {16},
      YEAR = {1977},
    NUMBER = {3},
     PAGES = {568--576},
      ISSN = {0024-6107},
   MRCLASS = {28A65},
  MRNUMBER = {MR0476995 (57 \#16540)},
MRREVIEWER = {Manfred Denker},
}

@misc{BanajiRutar24,
      title={Lower box dimension of infinitely generated self-conformal sets}, 
      author={Amlan Banaji and Alex Rutar},
      year={2024},
      eprint={2406.12821},
      archivePrefix={arXiv},
      primaryClass={math.DS},
      url={https://arxiv.org/abs/2406.12821}, 
      HOWPUBLISHED={Preprint, available at https://arxiv.org/abs/2406.12821},
}

@misc{RutarCFNotes26,
      title={Multi-scale properties of continued fraction sets}, 
      author={Alex Rutar},
      year={2026},
      eprint={2606.07139},
      archivePrefix={arXiv},
      primaryClass={math.NT},
      url={https://arxiv.org/abs/2606.07139}, 
      HOWPUBLISHED={Preprint, available at https://arxiv.org/abs/2606.07139},
}

\end{document}